\documentclass[12pt,table]{amsart}
\pdfoutput=1
\usepackage{underscore} % avoids \_ in textmode
\usepackage{mathscinet}
\usepackage[T1]{fontenc} % encoding
\usepackage{tikz} % code figures
\usepackage{tikz-cd} % commutative diagrams
\usetikzlibrary{matrix,arrows,decorations.pathmorphing} % special decorations for tikzmain.tex
\usepackage{amssymb, amsmath, mathtools, amsfonts, amsthm} % ams    
\usepackage{graphicx} % optional arguments for \includegraphics
\usepackage{fullpage} % sets margins to be either 1 inch or 1.5 cm
\usepackage{caption} % customize the captions in floating environments
\usepackage[inline,shortlabels]{enumitem}
\setitemize{itemsep=4pt}
\setenumerate{itemsep=3pt}
\usepackage{multirow} % for merging cells in tables
\usepackage{xcolor}
\usepackage{hyperref}
\definecolor{headercolor}{RGB}{255,255,240}
\definecolor{mylinkcolor}{RGB}{0,0,255}
\definecolor{mycitecolor}{RGB}{169,169,169}
\definecolor{myurlcolor}{RGB}{255,20,147}
\hypersetup{colorlinks=true,urlcolor=myurlcolor,citecolor=mycitecolor,linkcolor=mylinkcolor,linktoc=page,breaklinks=true}
\usepackage{url} 
\usepackage{cleveref} % automatic referencing
\usepackage[numbers]{natbib} % bibliography style
\usepackage{setspace}
\usepackage{indentfirst} % to indent the first paragraph of a section
\usepackage{listings} % source-code printing
\usepackage[mathscr]{eucal} % nice caligraphy for sheaves and categories
\usepackage{manfnt} % street sign
\usepackage{colonequals} % correct way to write :=
\usepackage{float}
\usepackage{soul} % to strikethrough words: \st{}
\usepackage{marvosym} % at work symbol
\usepackage{placeins}

\newcounter{counter}[section]

\numberwithin{equation}{section} 

\newtheorem{theorem}{Theorem}[section]

\newtheorem{lemma}[theorem]{Lemma}
\newtheorem{corollary}[theorem]{Corollary}

\newtheorem{proposition}[theorem]{Proposition}

\theoremstyle{definition}

\theoremstyle{remark}

\newcommand{\QQ}{\mathbb{Q}} % for Rationals
\newcommand{\FF}{\mathbb{F}} % finite field
\newcommand{\paren}[1]{\left( #1 \right)} % parenthesis
\newcommand{\brk}[1]{\left\lbrace #1 \right\rbrace} % brackets
\newcommand{\floor}[1]{\left\lfloor #1 \right\rfloor} % floor function
\let\originalleft\left \let\originalright\right
\renewcommand{\left}{\mathopen{}\mathclose\bgroup\originalleft}
  \renewcommand{\right}{\aftergroup\egroup\originalright}

\DeclareMathOperator{\Res}{Res} % Residue
\newcommand{\avlink}[1]{\href{http://www.lmfdb.org/Variety/Abelian/Fq/#1}{\texttt{#1}}}

\usepackage[textsize=tiny]{todonotes}
\usepackage{xargs, xpatch}
\thanks{This paper is the result of a collaboration initiated during the workshop ``Number Theory in the Americas 2'' held at Casa Matem\'atica Oaxaca in September 2024, supported in part by BIRS-CMO, the Journal of Number Theory, the Clay Mathematics Institute, and the IMU Commission for Developing Countries. Kedlaya was additionally supported by NSF (grants DMS-2053473, DMS-2401536) and UC San Diego (Warschawski professorship).
Rama was additionally partially supported by ANII FCE\_1\_2023\_1\_176497.}

\begin{document}

\title{Bounds for the relative class number problem for function fields}

\author{Santiago Arango-Piñeros} 
\address{Department of Mathematics, Emory
University, Atlanta, GA 30322, USA}
\email{santiago.arango.pineros@gmail.com}
\urladdr{\url{https://sarangop1728.github.io/}}

\author{María Chara}
\address{Edificio Babini, Facultad de Ingeniería Quimica, Universidad Nacional del Litoral, Obispo Gelabert 2846, Santa Fe, S3000AKK, Argentina}
\email{charamaria@gmail.com}
\urladdr{\url{https://sites.google.com/view/maria-chara/}}

\author{Asimina S. Hamakiotes}
\address{Department of Mathematics, University of Connecticut, Storrs, CT 06269, USA}
\email{asimina.hamakiotes@uconn.edu}
\urladdr{\url{https://asiminah.github.io/}}

\author{Kiran S. Kedlaya}
\address{Department of Mathematics, University of California San Diego, 9500 Gilman Drive \#0112, La Jolla, CA 92122, USA}
\email{kedlaya@ucsd.edu}
\urladdr{\url{https://kskedlaya.org}}

\author{Gustavo Rama}
\address{Facultad de Ingeniería, Universidad de la República, Montevideo, Uruguay}
\email{grama@fing.edu.uy}
\urladdr{\url{https://www.fing.edu.uy/~grama/}}

\begin{abstract}
    We establish bounds on a finite separable extension of function fields in terms of the relative class number, thus reducing the problem of classifying extensions with a fixed relative class number to a finite computation. We also solve the relative class number two problem in all cases where the base field has constant field not equal to $\FF_2$.
\end{abstract}

\maketitle

% %%%%%%%%%%%%% TODOs %%%%%%%%%%%%
% \makeatletter
% \providecommand\@dotsep{5}
% \makeatother
% \listoftodos\relax

% %%%%%%%%%%%%%%%%%%%%%%%%%%%%%%%%%%

\section{Introduction}

The relative class number one problem for number fields (more specifically, totally imaginary quadratic extensions of totally real number fields) was introduced by Stark \cite{Stark} as a generalization of the classical Gauss class number one problem for imaginary quadratic number fields.
While this problem has attracted much attention and seen much progress,
it remains open even under assumption of the generalized Riemann hypothesis.

By contrast, the relative class number one problem for function fields was first considered by Leitzel--Madan \cite{LeitzelMadan} but is now fully resolved \cite{Kedlaya2022,Kedlaya2024-II,Kedlaya2024-III}.
It is therefore natural to consider the problem of classifying finite separable extensions of function fields (of curves over finite fields) with relative class number equal to a fixed integer $m > 1$. (We exclude inseparable extensions because \emph{every} purely inseparable extension of function fields has relative class number 1.) It is also natural to restrict to extensions which are either \emph{constant} (i.e., induced by an extension of the constant field) or \emph{purely geometric} (i.e., with no change in the constant field), as any extension can be interpreted as a constant extension followed by a purely geometric extension.

In this paper, we make two contributions to this problem.
First, we reduce the relative class number $m$ problem for function fields, for any fixed positive integer $m$, to an explicit finite computation by giving bounds on the relevant parameters.
Second, we solve the relative class number 2 problem for function fields in all cases where the base field has constant field not equal to $\FF_2$.

Our arguments depend on a number of calculations made using SageMath and Magma.
These are collected in a series of Jupyter notebooks available from our GitHub repository:
\begin{center}
    \texttt{\url{https://github.com/sarangop1728/twice-class-number}}.
\end{center}

\subsection{Bounds for the class number problem}

Our bounds for the relative class number $m$ problem for general $m$ take the following form.
\begin{theorem} \label{thm:bounds}
Let $F'/F$ be a finite separable extension of function fields of degree $d$ with relative class number $m$.
Let $g$ be the genus of $F$ and let $g'$ be the genus of $F'$. 
Assume either that $F'/F$ is constant and $d>1$, or that $F'/F$ is purely geometric and $g' > g$.
Let $\FF_q$ be the constant field of $F$.
Let $g$ be the genus of $F$. Let $g'$ be the genus of $F'$.
\begin{enumerate}
\item[(a)]
We have $q \leq (\sqrt{m}+1)^2$.
\item[(b)]
If $F'/F$ is constant,
then $d$ and $g$ are bounded by functions of the form $a+b \log m$
where $a,b$ are some explicit constants.
\item[(c)]
If $F'/F$ is purely geometric, then
$g$ and $g'$ are bounded by functions of the form $a+b \log m$
where $a,b$ are some explicit constants.
\end{enumerate}
\end{theorem}
\begin{proof}
For (a), see \Cref{lem:bound-q-geometric-case} in the purely geometric case and 
\Cref{lem:constant-bounds} in the constant case.
For (b), see \Cref{lem:constant-bounds} for the bound on $d$ and for the bound on $g$ when $q > 4$; \Cref{lem:constant-bound-genus-g34} for the bound on $g$ when $q \in \{3,4\}$; and \Cref{lem:constant-bound-genus-g2} for the bound on $g$ when $q=2$.
For (c), see \Cref{lem:geometric-bound-g-q-gt-4} for the bounds when $q>4$;
\Cref{lem:geometric-bound-g-q-34} when $q \in \{3,4\}$; and \Cref{lem:bound-geometric-F2} when $q=2$.
\end{proof}
As an aside, note that the bound on $q$ is best possible: it is equivalent to $m \geq (\sqrt{q}-1)^2$,
and by Deuring's theorem this ensures that there is an elliptic curve $E$ over $\FF_q$ with $\#E(\FF_q) = m$. This gives rise to two different extensions with relative class number $m$: the function field of $E$ over a rational function field,
and the quadratic constant extension of the function field of the quadratic twist of $E$.

The techniques used in the above arguments depend heavily on $q$, the size of the constant field of $F$. When $q > 4$, the bounds are easily derived by applying the Weil bounds to the Prym variety of the extension. When $q \in \{3,4\}$, we must factor in the existence of an elliptic curve of order 1; for the complementary factor one gets a Weil-type exponential lower bound on the order of the Prym variety in terms of its dimension (see \cref{lemma:exp-lower-bounds} for the version used here), and then one must argue that the contribution of the elliptic curve forces the base field $F$ to have ``too many'' places of degree 1 compared to its genus. When $q = 2$, there is no exponential lower bound available, but one can instead directly pass from the knowledge that $m$ is small to a lower bound on some weighted count of places of low degree.

For comparison, we draw attention to recent work of Kadets--Keliher \cite{KadetsKeliher}. Their aim is to solve the much harder problem of finding all extensions of function fields
where the relative class group has a fixed exponent; if successful this would also give bounds on the relative class number problem, but at present their methods are not quite able to achieve this. However, these methods are different enough from ours that it is worth investigating whether one can combine ideas for better results.

\subsection{The relative class number 2 problem}
Our second contribution is to solve the relative class number 2 problem for function fields in all cases where the base field has constant field not equal to $\FF_2$.
\begin{theorem}
Let $F'/F$ be a finite separable extension of function fields with relative class number $2$. 
Let $\FF_q$ be the constant field of $F$ and assume that $q > 2$.
Let $g$ be the genus of $F$. Let $g'$ be the genus of $F'$.
\begin{enumerate}
\item[(a)] In all cases we have $q \in \{3,4,5\}$.
\item[(b)] If $F'/F$ is a constant extension, then $[F':F] = 2$ and 
\[
(q,g) \in \{(5,1), (4,1), (3,1), (3,2)\}.
\]
\item[(c)] If $F'/F$ is a purely geometric extension (that is, $\FF_q$ is also the constant field of $F'$) and $g=0$, then 
\[
(q,g') \in \{(5,1), (4,1), (3,1), (3,2)\}.
\]
\item[(d)] If $F'/F$ is a purely geometric extension and $g=1$, then 
\[
(q,g') \in \{ (5,2), (4,2), (3,2), (3,3), (3,4)\}.
\]
\item[(e)] If $F'/F$ is a purely geometric extension and $g \geq 2$, then 
\[
(q,g,g') \in \{(5,2,3), (4,2,3), (4,2,4), (4,3,5), (3,2,3), (3,2,4), (3,3,5), (3,4,7)\}.
\]
Moreover, $[F':F] = 2$ except in one case where $(q,g,g') = (3,2,4)$ and $F'/F$ is a noncyclic cubic extension.
\end{enumerate}
\end{theorem}

A more precise statement can be found in \Cref{prop:summary},
which also includes a description of the extensions that occur in cases (b), (c), and (d);
the proof of this statement occupies \S\ref{sec:q=5} (for $q > 4$), \S\ref{sec:q=4} (for $q=4$), and \S\ref{sec:q=3} (for $q=3$).
The extensions that occur in case (e) are listed in \S\ref{sec:explicit calculations for geometric}.

As in \cite{Kedlaya2022,Kedlaya2024-II,Kedlaya2024-III}, the argument here mostly involves work over $\FF_3$ and $\FF_4$, including a number of exhaustive searches for cyclic covers of particular curves (using explicit class field theory as implemented in Magma) and a few arguments to rule out most noncyclic covers of degree 3 or 4. One advantage over the case $m=1$ is that we can easily rule out cyclic covers of prime degree greater than 2 (\Cref{lem:prime-deg-congruence-geometric}). One disadvantage is that we do have to account for one noncyclic cubic cover; fortunately, we end up with enough circumstantial information about the cover that it is relatively straightforward to prove its existence and even to identify it explicitly (\Cref{lem:noncyclic cover}).

The case of the relative class number 2 problem in which the base field has constant field $\FF_2$ is expected to be substantially harder, due to the large genera of curves involved (although again there is no need to worry about cyclic covers of odd prime degree). To put this in context, recall that in \cite{Kedlaya2022} one finds an example of a quadratic extension of relative class number 1 where the base field has genus 7, and the arguments in \cite{Kedlaya2024-III} to rule out further such extensions require a restricted census over the $\FF_2$-points of the moduli space of curves of genus 7.  It is likely that a full resolution of the relative class number 2 problem in full would require at least a restricted census of curves of even larger genus over $\FF_2$, and it is unclear whether such a census would be feasible.

\subsection*{Acknowledgments}

Thanks to Carel Faber, Everett Howe, and David Roe for their help with generating data on Jacobians of curves and importing the data into LMFDB. Thanks also to Edgar Costa for arranging access to computing resources of the Simons Collaboration on Arithmetic Geometry, Number Theory, and Computation.

\section{General facts}

Given a finite separable extension of function fields $F'/F$, we write $C,C'$ for the curves corresponding to $F,F'$; $q,q'$ for the orders of the base fields of $C,C'$; $g,g'$ for the genera of $C,C'$; and  $h_F,h_{F'}$ for the class numbers of $F,F'$. We write $J = J(C)$ and $J' = J(C')$ for the Jacobians of $C,C'$, so that $\#J(\FF_{q}) = h_F$ and $\#J'(\FF_{q'}) = h_{F'}$.

We say that $F'/F$ is \emph{constant} if $F' = F \cdot \FF_{q'}$,
and \emph{purely geometric} if $q' = q$. If neither occurs, then $F'$ is a purely geometric extension of $F \cdot \FF_{q'}$, which is a constant extension of $F$; it thus suffices to study relative class number problems in these cases.

Let $A$ denote the Prym variety associated to the covering $C' \to C$, i.e., the reduced connected component of the identity of the kernel of the map 
\[
\Res_{\FF_{q'}/\FF_q}(J') \to J,
\]
where $\Res$ denotes the Weil restriction. We then have $\dim(A) = g'-g$ and $\#A(\FF_q) = h_{F'/F}$. Moreover, 
\begin{equation}
    \label{eq:h=P(1)}
    h_{F'/F} = P_A(1) = \prod_{\alpha}(1-\alpha),
\end{equation}
where the product ranges over the $q$-Frobenius eigenvalues of $A$.
See \cite[Section 2.6]{AchterCasalaina-Martin23} for generalities on Prym schemes.

For $\bullet$ a curve or abelian variety,
let $T_{\bullet,q^n}$ be the trace of the $q^n$-power Frobenius on $\bullet$
(acting on the first \'etale cohomology with coefficients in $\QQ_\ell$ for any prime $\ell$ which is nonzero in the base field).
Let $P_\bullet(T)$ denote the characteristic polynomial of Frobenius, which we also call the \emph{Weil polynomial} associated to $\bullet$.

\subsection{Bounds on Prym varieties}
\label{subsec:bounds-on-prym}

The fundamental tool in studying the relative class number for function fields is the Weil bound for the Prym variety $A$, plus some refinements for small $q$. See \cite{Kadets2021}, and the references therein, for a general reference on this topic.

\subsubsection{$q > 2$}
\begin{lemma}[Exponential lower bounds]
    \label{lemma:exp-lower-bounds}
    Let $B$ be an abelian variety over $\FF_q$ for $q > 2$
    with no isogeny factor in the isogeny classes \avlink{1.3.ad}, \avlink{1.4.ae}. Then 
    \begin{equation}
        \label{eq:exp-lower-bounds}
        \#B(\FF_q) \geq b_q^{\dim(B)},
    \end{equation}
    where $b_3 = 1.359$, $b_4 = 2$, and $b_q = (\sqrt{q}-1)^2$ for $q > 4$. 
\end{lemma}

\begin{proof}
    We immediately reduce to the case where $B$ is simple and not in one of the indicated isogeny classes.
    When $q > 4$, the result follows from the Weil bound. When $q =3,4$, the result follows from \cite[Theorem 3.2]{Kadets2021} with the given exceptions. 
\end{proof}

For $q=3,4$, we will use the nature of the exceptional cases in  Lemma~\ref{lemma:exp-lower-bounds} to establish that when $A$ has small order, its Frobenius traces must be large. To wit, the Prym variety decomposes as
\begin{equation*}
    A \sim E^{\dim(A)-\dim(B)}\times B,
\end{equation*}
where $E$ an elliptic curve of order 1 and $B$ is an abelian variety of order $m$ not containing $E$ as a factor. It follows from this isogeny decomposition that 
\begin{align}
    \label{eq:trace-lower-bound-q34}
    T_{A,q} &= (\dim(A) -\dim(B))T_{E,q} + T_{B,q} \\
    &\geq (\dim(A) - \dim(B))q - 2\dim(B)\sqrt{q} \notag \\
    &\geq q\dim(A)-(q +2\sqrt{q})\dim(B), \notag
\end{align}
where the middle inequality follows from $T_{E,q} = q$ and the Weil bound on the trace of $B$. 
By Lemma~\ref{lemma:exp-lower-bounds} we have
\begin{equation}
    \label{eq:Kadets-bounds}
    \dim(B) \leq \frac{\log(\#A(\FF_q))}{\log(b_q)},
\end{equation}
where $b_3 = 1.359$ and $b_4 = 2$.

\subsubsection{$q=2$}

This approach cannot be directly adapted to $q=2$ because there are infinitely many simple abelian varieties of order 1 over $\FF_2$ \cite{MadanPal}, so there is no hope of working with them concretely.
Instead, we give a lower bound for a suitable linear combination of traces of an abelian variety depending on its order. See \cite{Kedlaya-distribution} for a more detailed discussion.

\begin{lemma} \label{lemma:lower-bound-F2-traces}
    Let $A$ be an abelian variety over $\mathbb{F}_2$. Then
    \begin{gather} \label{eq:lower-bound-F2-traces}
    T_{A,2} + 0.2456 \cdot T_{A,2^2} +  0.0887 \cdot T_{A,2^3} +0.0295 \cdot T_{A,2^4} \\ \nonumber
    \geq 1.1691 \cdot \dim(A) - 1.9169 \cdot \log \#A(\FF_2).
    \end{gather}
\end{lemma}
\begin{proof}
    By Equation~\eqref{eq:h=P(1)} we can write 
    \[
    \#A(\FF_2) = P_A(1) = \prod_{i=1}^{\dim(A)} (1 - \alpha_i)(1 - \overline{\alpha}_i) = 
    \prod_{i=1}^{\dim(A)} (3 - \beta_i), 
    \]
    where the pairs $\alpha_i, \overline{\alpha_i}$ run over conjugate pairs of Frobenius eigenvalues of $A$
    and $\beta_i := \alpha_i + \overline{\alpha_i} \in [-2\sqrt{2}, 2\sqrt{2}]$.
    In this notation,
    \begin{align*}
        T_{A,2} &= \sum_{i=1}^{\dim(A)} \beta_i, \\
        T_{A,2^2} &= \sum_{i=1}^{\dim(A)} (\beta_i^2 - 4), \\
        T_{A,2^3} &= \sum_{i=1}^{\dim(A)} (\beta_i^3 - 6\beta_i), \\
        T_{A,2^4} &= \sum_{i=1}^{\dim(A)} (\beta_i^4 - 8\beta_i^2 + 8).
    \end{align*}
    The claim then follows from the fact that for $x \in [-2\sqrt{2}, 2\sqrt{2}]$,
    \begin{equation}
        \label{eq:positive}
        x + 0.2456 (x^2-4) + 0.0887 (x^3-6x) + 0.0295 (x^4-8x+8) \geq 1.1691 - 1.9169 \log (3-x),
    \end{equation}
    as may be verified rigorously using interval arithmetic (the difference between the two sides is bounded below by 0.0004). See the file \href{https://github.com/sarangop1728/twice-class-number/blob/main/lemma2-2.ipynb}{\texttt{lemma2-2.ipynb}}. \qedhere

    \begin{figure}[H]
    \centering
    \includegraphics[width=0.4\linewidth]{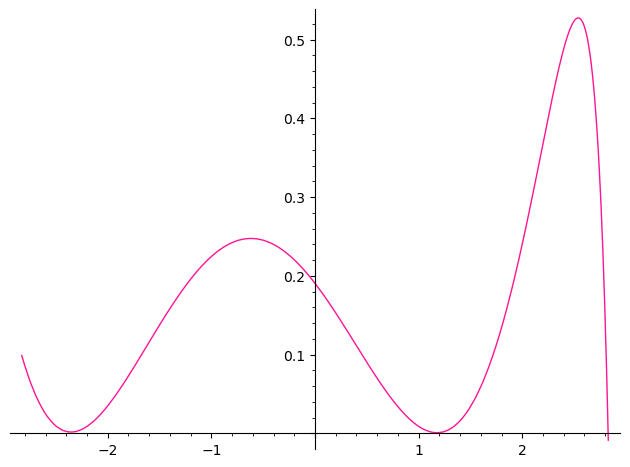}
    \caption{Positivity of the function defined by inequality (\ref{eq:positive}).}
    \label{fig:enter-label}
\end{figure}
\end{proof}

By comparison, when $\#A(\FF_2) = 1$, \cite[Lemma~5.6]{Kedlaya2022} yields
\[
T_{A,2} + 0.2314 \cdot T_{A,2^2} + 0.093 \cdot T_{A,2^3} + 0.044 \cdot T_{A,2^4} \geq 1.2766 \cdot \dim(A);
\]
this is slightly better in practice than specializing Lemma~\ref{lemma:lower-bound-F2-traces} to the case where $\#A(\FF_2) = 1$,
but the proof technique does not extend to $\#A(\FF_2) > 1$.

\subsection{Bounds on point counts}

We will play bounds off against upper bounds on rational points derived from nonnegativity constraints. See \cite[Chapter 6]{Serre20-rational-points} for a general reference on this topic.

\begin{lemma}
    \label{lemma:points-bounds}
    Let $C$ be a curve of genus $g$ over $\FF_q$, for $q \in\brk{3,4}$. Then
\begin{equation}
        \#C(\FF_q) \leq c_1g + c_0
\end{equation}
for the following values of the constants
\begin{table}[ht]
\setlength{\arrayrulewidth}{0.2mm} 
    \setlength{\tabcolsep}{5pt}
    \renewcommand{\arraystretch}{1.2}
    \centering
\begin{tabular}{|c|c|c|}
\hline
\rowcolor{headercolor} 
$q$ & $c_1$    & $c_0$   \\ \hline
%$2$ & $0.6272$ & $9.562$ \\ \hline
$3$ & $1.91$  & $5.52$ \\ \hline
%$3$ & $1.484$  & $8.026$ \\ \hline
$4$ & $2.31$  & $7.31$ \\ \hline
\end{tabular}
\end{table}
\end{lemma}
\begin{proof}
    We adapt \cite[Lemma 4.1]{Kedlaya2022} using the following values of $x_1$ and $x_2$.
    \begin{itemize}
        \item[($q=3$)] We use $x_1 = 0.932$, $x_2 = 0.0344$.%for the first bound and $x_1 = 0.85$, $x_2 = 0.182$, $x_3=0.084$ for the second bound.
        \item[($q=4)$] We use $x_1 = 0.8$, $x_2 = 0.02$. \qedhere
    \end{itemize}
\end{proof}

For explicit computations, we supplement with some upper bounds for the number of points $\#C(\FF_q)$ for a curve $C$ of genus $g$, taken from \href{https://manypoints.org/Default.aspx}{manYPoints} \cite{manypoints}.

\begin{table}[h!] 
\setlength{\arrayrulewidth}{0.2mm} 
    \setlength{\tabcolsep}{5pt}
    \renewcommand{\arraystretch}{1.2}
    \centering
\begin{tabular}{|c|c|c|c|c|}
\hline
\rowcolor{headercolor} 
$g$ & $q=3$ & $q=4$ & $q=3^2$ & $q=4^2$ \\
 \hline
$1$ & $7$ & $9$ & $16$  & $25$\\  \hline
$2$ & $8$ & $10$ & $20$ & $33$  \\  \hline
$3$ & $10$ & $14$ & $28$ & $38$ \\  \hline
$4$ & $12$ & $15$ & $30$ & $45$ \\  \hline
$5$ & $13$ & $17$ & $35$ & $53$  \\  \hline
$6$ & $14$ & $20$ & $38$ & $65$ \\  \hline
$7$ & $16$ & $21$ & $43$ & $69$ \\  \hline
$8$ & $18$ & $23$ & $46$ & $75$  \\  \hline
$9$ & $19$ & $26$ & $50$ & $81$ \\ \hline
\end{tabular}
\caption{Bound for $\#C(\FF_q)$ for a curve $C$ of genus $g$.}\label{table:many-points}
\label{tab:datos}
\end{table}

The following lemma, when compared with Lemma~\ref{lemma:lower-bound-F2-traces}, will give rise to bounds on the invariants of an extension $F'/F$ in the case that $q = 2$. See Lemma~\ref{lem:bound-geometric-F2}.
% For $q=2$, in order to compare against Lemma~\ref{lemma:lower-bound-F2-traces},
% we use a bound including some point counts over extension fields.

\begin{lemma} \label{lem:bounds-on-point-counts-F2}
Let $C$ be a curve of genus $g$ over $\FF_2$. 
Let $a_r$ be the number of degree-$r$ places of $C$. Then
\[
a_1 + 0.3542 \cdot 2a_2 + 0.1524 \cdot 3a_3 + 0.0677 \cdot 4a_4 \leq 0.7365 g + 6.5441.
\]
\end{lemma}
\begin{proof}
     We use the parameters $x_2 = 0.8$ and $x_3 = 0.6$. See \cite[Lemma~4.2]{Kedlaya2022}. The calculations can be found in the file \href{https://github.com/sarangop1728/twice-class-number/blob/main/linear_programming.ipynb}{\texttt{linear-programming.ipybn}}.
\end{proof}

\subsection{Constant case}

For $F'/F$ constant of degree $d$,
\[
P_{C'}(T^d) = P_C(T) P_C(\zeta_d T) \cdots P_C(\zeta_d^{d-1} T),
\]
where $\zeta_d$ is a primitive $d$-th root of unity. Consequently,
\begin{equation} \label{eq:trace-constant-Prym}
T_{A,q^i} = \begin{cases} -T_{C,q^i} & i \not\equiv 0 \pmod{d}, \\
(d-1) T_{C,q^i} & i \equiv 0 \pmod{d}.
\end{cases}
\end{equation}

When studying the relative class number $m$ problem for $m>1$, the possible degrees of constant extensions are heavily restricted. We first examine this in the prime case.
\begin{lemma} \label{lem:prime-deg-congruence}
    If $F'/F$ is a constant extension of prime degree $\ell > 2$ with $g_F > 0$, then $h_{F'/F} \equiv 0, 1 \pmod \ell$.
\end{lemma}
\begin{proof}
When $F'/F$ is a constant extension, we may obtain $J'$ from $J$ by base extension from $\FF_q$ to $\FF_{q^\ell}$,
and so $h_{F'/F} = \#J(\FF_{q^\ell})/\#J(\FF_q)$. Now,  
\begin{equation*}
    \#J(\FF_{q^\ell})/\#J(\FF_q) = P_{C'}(1)/P_C(1) = P_C(\zeta)P_C(\zeta^2)\cdots P_C(\zeta^{\ell-1})
\end{equation*}
is a norm of the field $\QQ(\zeta)$, where $\zeta$ is a primitive $\ell$-th root of unity. We conclude that $h_{F'/F} \equiv 0, 1 \pmod{\ell}$. 
\end{proof}

We then record some consequences for the case $m=2$.
\begin{corollary}
    \label{cor:constan-prime-deg}
    Let $F'/F$ be a constant extension with $g_F > 0$, relative class number $2$, and prime degree $\ell = [F':F]$. Then $\ell = 2$.
\end{corollary}

\begin{corollary}
 \label{cor:constan-composite-deg}
 Let $F'/F''/F$ be a tower of proper constant extensions, with corresponding fields of constants $\FF_{q'}/\FF_{q''}/\FF_q$. Assume also that $g_F >0$ and that $h_{F'/F} = 2$. Then:
 \begin{enumerate}[label=(\alph*)]
     \item If $h_{F'/F''} = 1$, then $q'' \leq 4$.
     \item If $h_{F'/F''} = 2$, then $q \leq 4$. 
 \end{enumerate}
\end{corollary}
\begin{proof}
    Recall that $h_{F'/F} = h_{F'/F''}\cdot h_{F''/F}$.
    \begin{enumerate}[label=(\alph*)]
     \item If $h_{F'/F''} = 1$, then \cite[Theorem 1.1]{Kedlaya2022} gives that $q'' \leq 4$.
     \item If $h_{F'/F''} = 2$, then $h_{F''/F} = 1$ and \cite[Theorem 1.1]{Kedlaya2022} gives that $q \leq 4$. \qedhere
 \end{enumerate}
\end{proof}

\subsection{Purely geometric case}

A crucial observation that we will need is the following upper bound on the trace of the Prym variety \cite[Equation (5.3)]{Kedlaya2022}.
\begin{lemma}[Prym trace less than rational points on curve]
\label{trace-bound}
Suppose that $F'/F$ is purely geometric. Then for every $r \geq 1$,
    \begin{equation}
    \label{eq:trace-less-than-points-with-moebius}
        \sum_{n \mid r}\mu\paren{\frac{r}{n}}  T_{A,q^n} \leq  ra_r,
    \end{equation}    
where $\mu$ is the M\"obius function, and $a_r$ is the number of degree-$r$ places of $C$. In particular,
    \begin{equation}
    \label{eq:trace-less-than-points}
        T_{A,q^r} \leq \#C(\FF_{q^r}) = \sum_{n\mid r} na_n.
    \end{equation}    
\end{lemma}
\begin{proof}
    From the short exact sequence that defines $A$, we have that $T_{C',q^r} = T_{C,q^r} + T_{A,q^r}$. Therefore,
    \begin{align*}
        0 \leq \#C'(\FF_{q^r}) &= q^r + 1 - T_{C',q^r} \\
        &= q^r + 1 - T_{C,q^r} - T_{A,q^r} \\
        &= \#C(\FF_{q^r}) - T_{A,q^r}. 
    \end{align*}
    Similarly, $\sum_{n \mid r} \mu\paren{\frac{r}{n}} \#C'(\FF_{q^n})$ equals $r$ times the number of degree-$r$ places of $C'$, and hence is also nonnegative.
    That means that
    \begin{align*}
    0 &\leq \sum_{n \mid r} \mu\paren{\frac{r}{n}} (\#C(\FF_{q^n}) - T_{A,q^n}) \\
    &= \sum_{n \mid r} \mu\paren{\frac{r}{n}} \#C(\FF_{q^n}) -
    \sum_{n \mid r} \mu\paren{\frac{r}{n}} T_{A,q^n} \\
    &= ra_r - \sum_{n \mid r} \mu\paren{\frac{r}{n}} T_{A,q^n}. \qedhere
    \end{align*}
\end{proof}

When $g \geq 2$ and $d=2$, from \cite[(7.7)]{Kedlaya2022}  and \Cref{eq:trace-less-than-points} we have
\begin{equation} 
\label{eq:trace-upper-bound-deg-2}
T_{A,q^{2i}} - t \leq \#C(\FF_{q^{2i}}) - 2\#C(\FF_{q^i}),    
\end{equation}
where $t$ is the number of geometric ramification points. By Riemann--Hurwitz, $t \leq 2(g'-(d-1)(g-1))$
if $q$ is odd and $t \leq g'-(d-1)(g-1)$ if $q$ is even.

For cyclic extensions, we have an analogue of Lemma~\ref{lem:prime-deg-congruence}.
\begin{lemma} \label{lem:prime-deg-congruence-geometric}
    Let $F'/F$ be a purely geometric extension which is cyclic of prime degree $\ell > 2$.
    Then $h_{F'/F} \equiv 0, 1 \pmod \ell$.
\end{lemma}
\begin{proof}
Let $\zeta$ be a primitive $\ell$-th root of unity.
Form the \'etale cohomology groups $H^1(C)$ and $H^1(C')$ with coefficients in $\overline{\QQ_{\ell'}}$ for some prime $\ell'$ not equal to the characteristic of $F$.
The quotient $H^1(C')/H^1(C)$, as a representation of the cyclic group $\mathrm{C}_\ell$,
splits into eigenspaces for the $\ell-1$ nontrivial characters of $\mathrm{C}_\ell$.
Since Frobenius commutes with the action of $\mathrm{C}_\ell$, it respects the eigenspace decomposition;
moreover, we may obtain all of the Frobenius eigenvalues appearing in $H^1(C')/H^1(C)$ by taking the eigenvalues in one eigenspace together with their Galois conjugates over $\QQ(\zeta)$.
Consequently, the product of $1-\alpha$ as $\alpha$ runs over the Frobenius eigenvalues in one eigenspace is an element of $\QQ(\zeta)$ with norm $h_{F'/F} = P_A(1)$. This proves the claim.
\end{proof}

\subsection{Noncyclic extensions of degree 3 and 4}
\label{subsec:noncyclic-description}

Following \cite{Kedlaya2024-II},
we enumerate the possible shapes of noncyclic, purely geometric extensions $F'/F$ of function fields of degree 3 or 4.
In the following discussion, let $F''/F$ be the Galois closure of $F'/F$.

For $d=3$, if $F'/F$ is not cyclic, then $F''/F$ has Galois group $S_3$ and so contains a quadratic subextension $F'''/F$ (the \emph{quadratic resolvent} of the original cubic extension). This extension may be either constant or purely geometric. However, if there exists a degree-1 place of $F$ which does not lift to a degree-1 place of $F'$ 
(which is automatic if $\#C'(\FF_q) < \#C(\FF_q)$),
then as in \cite[Lemma~4.5]{Kedlaya2024-II}, we see that the quadratic resolvent must be purely geometric.
In this case, $F'''/F$ corresponds to a covering $C''' \to C$ whose Prym $A'$ has the property that 
\[
J(C'') \sim J \times A^2 \times A'.
\]

For $d=4$, one possibility is that $F'/F$ admits an intermediate subfield (this includes the cyclic case).
Excluding this option leaves three other possibilities. One is that $F''/F$ has Galois group $A_4$;
by \cite[Lemma~4.5]{Kedlaya2024-II}, this case can be excluded by ensuring that some degree-1 place of $F$ lifts to a degree-4 place of $F'$.
In the other two cases, $F''/F$ has Galois group $S_4$ and again has a quadratic resolvent $F'''/F$ which can be either constant or purely geometric.

%%%%%%%%%%%%%%%%%%%%%%%%%%%%%%%%%%

\section{Bounds for the relative class number problem}
\label{sec:bounds}

In this section, we give bounds on a finite separable extension $F'/F$ in terms of the relative class number $m := h_{F'/F}$, in the constant case and the purely geometric case.

\subsection{Purely geometric extensions}

When $F'/F$ is purely geometric, the relevant statement is a bound on $g,g',q$ in terms of $m$.

\subsubsection{$q > 4$} We have the following bounds on $g$ and $g'$.

\begin{lemma} \label{lem:geometric-bound-g-q-gt-4}
    Suppose that $F'/F$ is purely geometric, $g'>g$, and $q >4$. Then
    \begin{enumerate}
        \item If $g \leq 1$, we have
        \begin{equation*}
            g' \leq 1 + \dfrac{\log(m)}{2\log(\sqrt{q}-1)} \leq 1 + \dfrac{\log(m)}{2\log(\sqrt{5}-1)} \leq 1 + 0.106 \log(m).
        \end{equation*}
        \item If $g > 1$ and $[F':F] = d$, then
        \begin{align*}
            g &\leq 1 + \dfrac{\log(m)}{2(d-1)\log(\sqrt{q}-1)} \leq 1 + \dfrac{\log(m)}{2(d-1)\log(\sqrt{5}-1)} \leq 1 + 0.106 \log(m), \\
            g' &\leq 1 + \dfrac{\log(m)}{\log(\sqrt{q}-1)} \leq 1 + \dfrac{\log(m)}{\log(\sqrt{5}-1)} \leq 1 + 0.212 \log(m).
        \end{align*}
    \end{enumerate}
\end{lemma}
\begin{proof}
    Since $F'/F$ is purely geometric, we  employ the exponential lower bounds on the number of points of Lemma~\ref{lemma:exp-lower-bounds}, to get an upper bound on the difference $g' - g$: 
    \begin{equation}
    \label{eq:Weil-lower-bound}
        g'-g = \dim(A)  \leq  \dfrac{\log(m)}{\log(b_q)}.
    \end{equation}
    From the Weil lower bound we take $b_q = (\sqrt{q}-1)^2$, and then we deduce that 
    \begin{equation*}
        g'-g = \dim(A) \leq \dfrac{\log(m)}{2\log(\sqrt{q}-1)}.
    \end{equation*}
    When $g \leq 1$, we obtain the upper bound in (1). When $g > 1$, Riemann--Hurwitz gives $g'-g \geq (d-1)(g-1)\geq(g-1)$, and therefore $g'\geq 2g-1$ or equivalently $g\leq \frac{1}{2}(g'+1)$.  Using this bound on $g$ on $$g'\leq \dfrac{\log(m)}{2\log(\sqrt{q}-1)}+g$$ we obtain (2). Explicitly, we have 
    \begin{align*}
        g'&\leq \dfrac{\log(m)}{2\log(\sqrt{q}-1)}+\frac{1}{2}(g'+1),\\
        2g'&\leq \dfrac{\log(m)}{\log(\sqrt{q}-1)}+g'+1,\\ 
        g'&\leq \dfrac{\log(m)}{\log(\sqrt{q}-1)}+1.
    \end{align*}
\end{proof}

We bound $q$ in a similar way.

\begin{lemma} \label{lem:bound-q-geometric-case}
    Suppose that $F'/F$ is purely geometric and $g' > g$. Then
    \begin{equation*}
        q \leq m + 1 + \floor{2\sqrt{m}}.
    \end{equation*}
\end{lemma}
\begin{proof}
    As in the proof of Lemma~\ref{lem:geometric-bound-g-q-gt-4} we have 
    \begin{equation*}
        1 \leq g'-g \leq \dfrac{\log(m)}{2\log(\sqrt{q}-1)}.
    \end{equation*}
    Solving this inequality for $q$ yields the result.
\end{proof}

\subsubsection{$q \in \brk{3,4}$} 
\begin{lemma} \label{lem:geometric-bound-g-q-34}
    Suppose that $F'/F$ is purely geometric, $g' > g$, and $q \in \brk{3,4}$.
    \begin{enumerate}
        \item If $g \leq 1$, we have
        \begin{align*}
            g' &\leq 1 + \frac1q\paren{c_1 + c_0 + (q+2\sqrt{q})\dfrac{\log(m)}{\log(b_q)}}, \\
            &\leq \begin{cases}
                3.4767 + 7.0244 \log(m) & q=3, \\
                3.405 + 2.8854 \log(m) & q=4.
            \end{cases}
        \end{align*}
        \item If $g > 1$ and $[F':F] = d$, then
        \begin{align*}
            g &\leq 1 + \dfrac{1}{q(d-1)-c_1}\paren{ c_1 + c_0 + (q+2\sqrt{q})\dfrac{\log(m)}{\log(b_q)}}, \\
            &\leq \begin{cases}
                7.8166 + 19.333 \log(m) & q = 3, \\
                6.6924 + 6.8294 \log(m)  & q=4.
            \end{cases} \\
            g' &\leq 1 + \dfrac{2}{q-c_1}\paren{ c_1 + c_0 + (q+2\sqrt{q})\dfrac{\log(m)}{\log(b_q)}}, \\
            &\leq \begin{cases}
                14.6332 + 38.666 \log(m) & q = 3, \\
                12.3848 + 13.6588 \log(m)  & q=4.
            \end{cases}
        \end{align*}
    \end{enumerate}
\end{lemma}
\begin{proof}
    From \Cref{eq:Kadets-bounds,lemma:points-bounds,eq:trace-lower-bound-q34} we have that:
    \begin{align*}
        c_1g + c_0 \geq \#C(\FF_q) \geq T_{A,q} &\geq q\dim(A) - (q+2\sqrt{q})\dim(B), \\
        &\geq q(g'-g) - (q+2\sqrt{q})\dfrac{\log(m)}{\log(b_q)}.
    \end{align*}
    Solving for $g'-g$ we obtain:
    \begin{equation*}
        g'-g \leq \frac1q\paren{c_1g + c_0 + (q+2\sqrt{q})\dfrac{\log(m)}{\log(b_q)}}.
    \end{equation*}
    If $g\leq 1$, then $g'-1\leq g'-g$ and we obtain (1). If $g > 1$, then Riemann--Hurwitz gives $(d-1)(g-1)\leq (g'-g)$ and we obtain (2);
    to minimize $g'$ we first take $d = 2$ and then $g' = 2g-1$.
\end{proof}

One can improve this bound by distinguishing the cases $d=2$ and $d>2$, 
using \Cref{eq:trace-upper-bound-deg-2} in the former case. We will use this strategy in the computations for $m=2$.

\subsubsection{$q=2$}
\label{subsubsec-bounds-F2}

In this case, we cannot separate the Prym variety in any useful way,
so we instead appeal to Lemma~\ref{lemma:lower-bound-F2-traces}
and Lemma~\ref{lem:bounds-on-point-counts-F2}.

\begin{lemma} \label{lem:bound-geometric-F2}
    Suppose that $F'/F$ is purely geometric, $g' > g$, and $q=2$.
    \begin{enumerate}
        \item If $g \leq 1$, we have
        \[
        g' \leq 8.9408 + 1.6397 \log(m).
        \]
        \item If $g > 1$ and $[F':F] = d$, then
        \begin{align*}
            g &\leq 1 + \frac{1}{d-1.8033} (8.9408 + 1.6397 \log (m)) 
            \leq 45.454 + 8.3661 \log(m), \\
            g'&\leq 1 + \frac{2}{2-1.8033} (8.9408 + 1.6397 \log (m)) 
            \leq 89.908 + 8.3661 \log(m).
        \end{align*}
    \end{enumerate}
\end{lemma}
\begin{proof}
By \Cref{eq:trace-less-than-points},
\begin{align*}
    1.2751 (0.7365g + 6.5441) &\geq 1.2751 (a_1 + 0.3542 \cdot 2a_2 + 0.1524 \cdot 3a_3 + 0.0677 \cdot 4a_4), \\
    & \geq 1.2751 \cdot a_1 + 0.2751 \cdot 2a_2 + 0.0877 \cdot 3a_3 + 0.0295 \cdot 4a_4, \\
    & = \#C(\FF_2) + 0.2456 \cdot \#C(\FF_{2^2}) + 0.0877 \cdot \#C(\FF_{2^3}) + 0.0295 \cdot \#C(\FF_{2^4}), \\
    & \geq T_{A,2} + 0.2456 \cdot T_{A,2^2} + 0.0877 \cdot T_{A,2^3} + 0.0295 \cdot T_{A,2^4}, \\
    & \geq 1.1691 (g'-g) - 1.9169 \log(m).
\end{align*}
Solving for $g'-g$ we obtain
\[
g'-g \leq 0.8033 g + 7.1375 + 1.6397 \log(m).
\]
If $g \leq 1$, then $g'-1 \leq g'-g$ and we obtain (1).
If $g > 1$, then Riemann--Hurwitz gives $(d-1)(g-1) \leq (g'-g)$ and we obtain (2);
    to minimize $g'$ we first take $d = 2$ and then $g' = 2g-1$.
\end{proof}

Again, one can improve this bound by distinguishing the cases $d=2$ and $d>2$, 
using \Cref{eq:trace-upper-bound-deg-2} in the former case; for $m=1$ this is the content of \cite[Lemma~10.1]{Kedlaya2022}.

%%%%%%%%%%%%%%%%%%%%%%%%%%%%%%%%%%

\subsection{Constant extensions}

When $F'/F$ is constant, the relevant statement is a bound on $d,g,q$ in terms of $m$.

\subsection{\texorpdfstring{$q > 4$}{q > 4}}

We apply the Weil bounds as follows.
\begin{lemma}
    \label{lemma:m-times-points}
    Let $J$ be an abelian variety over $\FF_q$, with $q>4$, and let $m$ be a positive integer. Suppose that $\#J(\FF_{q^d}) = m\cdot\#J(\FF_q)$ for some $d \geq 2$. Then
    \begin{equation*}
        \paren{\frac{q^{d/2}-1}{q^{1/2}+1}}^{2\dim (J)} \leq m.
    \end{equation*}
\end{lemma}
\begin{proof}
    We have
    \begin{align*}
        (q^{d/2}-1)^{2\dim(J)} &\leq \#J(\FF_{q^d}) \\
        &= m\cdot \#J(\FF_q) \leq m(q^{1/2}+1)^{2\dim (J)}, 
    \end{align*}
    from which we deduce the desired bound.
\end{proof}

We now apply \Cref{lemma:m-times-points} in the case when $J$ is the Jacobian of a curve $C / \FF_q$, corresponding to a constant extension of function fields $F'/F$ with relative class number $m$, so that
\begin{equation} \label{eq:lower-bound-constant-case}
     \left( \frac{q^{d/2}-1}{q^{1/2}+1}\right)^{2g} \leq m.
\end{equation}
Using this, we are able to find explicit bounds on $d$, $q$, and $g$, solely in terms of $m$.

\begin{lemma}
    \label{lem:constant-bounds}
    Let $F'/F$ be a constant extension of function fields of degree $d>1$ and relative class number $m$. Then
    \begin{align}
        d &\leq 2 \log_2 ((\sqrt{2}+1)\sqrt{m} +1), \\
        q &\leq (\sqrt{m} + 1)^2.
    \end{align}
    In addition, if $q > 4$, then
    \begin{equation}
       0 < g \leq \frac{\log(m)}{2 \log (\sqrt{5}-1)}.
    \end{equation}
\end{lemma}
\begin{proof}
To bound $d$, we apply \Cref{eq:lower-bound-constant-case} and extremize by taking $q=2, g=1$ to obtain
\[
m \geq \paren{\frac{2^{d/2}-1}{2^{1/2}+1}}^2.
\]
To bound $q$, we apply \Cref{eq:lower-bound-constant-case} and extremize by taking $d=2, g=1$ to obtain
\[
m \geq \paren{\frac{q-1}{q^{1/2}+1}}^2 = (\sqrt{q}-1)^2.
\]
To bound $g$,  we apply \Cref{eq:lower-bound-constant-case} and extremize by taking $d=2, q=5$ to obtain
\[
m \geq (\sqrt{5}-1)^{2g}. \qedhere
\]
\end{proof}

\subsubsection{$q \in \brk{3,4}$}

For $q \in \{3,4\}$, we must distinguish between the cases $d>2$,
for which Lemma~\ref{lemma:m-times-points} gives a nontrivial bound on $g$, and the case $d=2$, for which no such bound can exist because of the existence of elliptic curves with $m=1$. We handle these cases by adapting the argument from the purely geometric case, but now with the bounds on traces replaced with the explicit formula \Cref{eq:trace-constant-Prym}.

\begin{lemma} \label{lem:constant-bound-genus-g34}
    Let $F'/F$ be a constant extension of function fields of degree $d > 1$ and relative class number $m$. 
    \begin{enumerate}
        \item If $d > 2$ and $q = 3$, then 
        \[
        0 < g \leq 1.17 \log(m).
        \]
        \item If $d > 2$ and $q = 4$, then 
        \[
        0 < g \leq 0.6 \log(m).
        \]
        \item If $d=2$ and $q = 3$, then 
        \[
0 < g \leq 19.34 \log(m) + 1.4.        
        \]
        \item If $d=2$ and $q = 4$, then 
        \[
0 < g \leq 6.83 \log(m) + 1.37.
        \]
    \end{enumerate}
\end{lemma}
\begin{proof}
    For $d > 2$, we again apply \Cref{eq:lower-bound-constant-case} and extremize by taking $d=3$ to obtain
    \[
    m \geq \paren{\frac{q^{3/2}-1}{q^{1/2}+1}}^{2g}.
    \]
    For $d=2$,
we apply \Cref{lemma:points-bounds} and \Cref{eq:trace-lower-bound-q34,eq:Kadets-bounds} to write
\begin{align*}
c_1 g + c_0 &\geq \#C(\FF_q) \\
&=q+1-T_{C,q} \\
&= q+1+T_{A,q} \\
&\geq q+1+q g - (q+2\sqrt{q}) \frac{\log(m)}{\log(b_q)},
\end{align*}
and hence
\[
g \leq \frac{1}{q-c_1} \paren{(q + 2\sqrt{q}) \frac{\log(m)}{\log(b_q)}+c_0 -q-1}. 
\qedhere
\]
\end{proof}

\subsubsection{$q=2$}

For $q = 2$, we must distinguish between the cases $d>3$,
for which Lemma~\ref{lemma:m-times-points} gives a nontrivial bound on $g$, and the cases $d=2$ and $d=3$, where again we must argue in terms of bounds on point counts.

\begin{lemma} \label{lem:constant-bound-genus-g2}
    Let $F'/F$ be a constant extension of function fields of degree $d > 1$ and relative class number $m$. 
    \begin{enumerate}
        \item If $d > 3$ and $q = 2$, then 
        \[
        0 < g \leq 2.31 \log(m).
\]
        \item If $d = 3$ and $q = 2$, then 
        \[
        0 < g \leq 2.8971 + 1.254 \log(m).
\]
        \item If $d = 2$ and $q = 2$, then 
        \[
        0 < g \leq 5.7135 + 3.3642 \log(m).
\]
    \end{enumerate}
\end{lemma}
\begin{proof}
    For $d > 3$, we again apply \Cref{eq:lower-bound-constant-case} and extremize by taking $d=4$ to obtain
    \[
    m \geq \paren{\frac{2^{4/2}-1}{2^{1/2}+1}}^{2g}.
    \]
    For $d = 3$, by Lemma~\ref{lemma:lower-bound-F2-traces}, Lemma~\ref{lem:bounds-on-point-counts-F2}, and \eqref{eq:trace-constant-Prym}, we have
    \begin{align*} 
        0.8085g + 3.2728
        &\geq 1.0977 (0.7365g + 6.5441) - 3.9107 \\
        &\geq 1.0977 (a_1 + 0.3542 \cdot 2a_2 + 0.1524 \cdot 3a_3 + 0.0677 \cdot 4a_4) - 3.1329 \\
        &\geq 1.0977 \cdot a_1 + 0.2751 \cdot 2a_2 - 0.1774 \cdot 3a_3 + 0.0295 \cdot 4a_4 - 3.1329 \\
        &= (\#C(\FF_2)-2-1) + 0.2456 (\#C(\FF_{2^2})-2^2-1) \\
        & \quad + 0.1774 (2^3+1-\#C(\FF_{2^3})) + 0.0295 (\#C(\FF_{2^4}) - 2^4-1) \\
        &= -T_{C,2} - 0.2456 \cdot T_{C,2^2} + 0.0887 \cdot 2 T_{C,2^3} - 0.0295 \cdot T_{C,2^4} \\        
        &= T_{A,2} + 0.2456 \cdot T_{A,2^2} + 0.0887 %\sapedit{0.0877/2} %0.0877
        \cdot T_{A,2^3} + 0.0295 \cdot T_{A,2^4} \\
        &\geq 1.1691 \dim(A) - 1.9169 \log \#A(\FF_2) \\
        &= 1.1691 \cdot 2g - 1.9169 \log(m).
    \end{align*}
    For $d = 2$, we have similarly
    \begin{align*}
        0.5993 g + 3.2555
        &\geq 0.8136 (0.7365g + 6.5441) - 2.0688 \\
        &\geq 0.8136 (a_1 + 0.3542 \cdot a_2 + 0.1524 \cdot 3a_3 + 0.0677 \cdot 4a_4) - 2.0688 \\
        &= 0.8136 \cdot a_1 - 0.2751 \cdot 2a_2 + 0.0887 \cdot 3a_3 - 0.0295 \cdot 4a_4 - 2.0688 \\
        &= (\#C(\FF_2)-2-1) + 0.2456 (2^2+1-\#C(\FF_{2^2})) \\
        &  \qquad  + 0.0887 (\#C(\FF_{2^3}) - 2^3 - 1) 
           + 0.0295 (2^4+1-\#C(\FF_{2^4})) \\
        &= -T_{C,2} + 0.2456 \cdot T_{C,2^2} - 0.0887 \cdot T_{C,2^3} + 0.0295 \cdot T_{C,2^4} \\        
        &= T_{A,2} + 0.2456 \cdot T_{A,2^2} + 0.0887 \cdot T_{A,2^3} + 0.0295 \cdot T_{A,2^4} \\
        &\geq  1.1691 \dim(A) - 1.9169 \log \#A(\FF_2) \\
        &= 1.1691 g - 1.9169 \log(m). \qedhere
    \end{align*}
\end{proof}

%\newpage
\section{\texorpdfstring{$m=2$}{m=2}: an overview}
\label{sec:relative class number two}

We next turn to the case $m=2$. In the next three sections, we give a proof of the following result
summarized in Table~\ref{table:bounds table}.
\begin{proposition} \label{prop:summary}
Let $F'/F$ be a finite separable extension of function fields with $h_{F'}/h_F = 2$, which is either constant or purely geometric; assume in addition that $q > 2$.
\begin{itemize}
\item We have $q \leq 5$.
 \item
 If $F'/F$ is constant, then $d = 2$ and the isogeny class of $J$ is one of \avlink{1.5.e} (for $q=5$),
 \avlink{1.4.d} (for $q=4$), or \avlink{1.3.d} or \avlink{2.3.e_i} (for $q=3$).
 \item
 If $F'/F$ is purely geometric and $g=0$, then the isogeny class of $J'$ is one of \avlink{1.5.ae} (for $q=5$),
 \avlink{1.4.ad} (for $q=4$), or \avlink{1.3.ad} or \avlink{2.3.ae_i} (for $q=3$).
 \item
 If $F'/F$ is purely geometric and $g=1$, then the isogeny classes of $J'$ and $C$ form a pair listed in Table~\ref{tab:geometric-q-C-elliptic}.
 \item
  If $F'/F$ is purely geometric, $g \geq 2$, and $d=2$, then $(q,g,g')$ must equal a tuple indicated in Table~\ref{table:bounds table}.
\item
 If $F'/F$ is purely geometric, $g \geq 2$, and $d>2$, then $F'/F$ is equal to a certain noncyclic extension with $q=3, d=3, g=2, g'=4$; see Lemma~\ref{lem:noncyclic cover}.
\end{itemize}
\end{proposition}
\begin{proof}
The case $q > 4$ is treated in section \ref{sec:q=5};
the case $q=4$ is treated in section \ref{sec:q=4};
the case $q=3$ is treated in section \ref{sec:q=3}.
More precise cross-references can be found in Table~\ref{table:bounds table}.
\end{proof}

Proposition~\ref{prop:summary} reduces the relative class number 2 problem for function fields with $q>2$ to an exhaust over cyclic covers for some particular values of $q,g,d,g,g'$ with $g \geq 2$. We discuss this exhaust in section \ref{sec:explicit calculations for geometric}.

\begin{table}[ht]
\setlength{\arrayrulewidth}{0.2mm} 
    \setlength{\tabcolsep}{5pt}
    \renewcommand{\arraystretch}{1.2}
\begin{tabular}{|ll|c|c|c|}
\hline
\rowcolor{headercolor} 
\multicolumn{2}{|l|}{Case}                                        & $g $ & $g'$ & Proof \\ \hline
\multicolumn{1}{|l|}{\multirow{2}{*}{$q>4$}}   & constant         & $1$      & -        & \ref{sec:subsec-q=5-constant-case}      \\ \cline{2-5} 
\multicolumn{1}{|l|}{}                         & purely geometric & $2$      & $3$      &  \ref{sec:subsec-q=5-geometric-case}     \\ \hline
\multicolumn{1}{|l|}{\multirow{5}{*}{$q = 4$}} & constant         & $1$      & -        &  \ref{sec:subsec-q=4-constant-case}     \\ \cline{2-5} 
\multicolumn{1}{|l|}{}                         & $g\leq 1$        & $1$      & $2$      &  \ref{sec:subsec-q=4-geometric-g=0-case}, \ref{sec:subsec-q=4-geometric-g=1-case}
     \\ \cline{2-5} 
\multicolumn{1}{|l|}{}                         & \multirow{2}{*}{$d = 2$}           & $2$      & $3$--$4$ & \multirow{2}{*}{\ref{sec:subsec-q=4-geometric-d=2-case}}      \\ \cline{3-4} 
\multicolumn{1}{|l|}{}                         &          & $3$      & $5$      &       \\ \cline{2-5}
\multicolumn{1}{|l|}{}                         & $d > 2$         & -      & -      &    \ref{sec:subsec-q=4-geometric-d=3-case}   \\ \hline
\multicolumn{1}{|l|}{\multirow{7}{*}{$q=3$}}   & constant         & $2$      & -        &  \ref{sec:subsec-q=3-constant-case}
     \\ \cline{2-5} 
\multicolumn{1}{|l|}{}                         & $g\leq 1$        & $1$      &    $4$      & \ref{sec:subsec-q=3-geometric-g=0-case}, \ref{sec:subsec-q=3-geometric-g=1-case}      \\ \cline{2-5} 
\multicolumn{1}{|l|}{}                         & \multirow{3}{*}{$d=2$}            & $2$      & $3$--$4$     &  \multirow{3}{*}{\ref{sec:subsec-q=3-geometric-d=2-case}}     \\ \cline{3-4} 
\multicolumn{1}{|l|}{}                         &                  & $3$      & $5$--$6$     &       \\ \cline{3-4} 
\multicolumn{1}{|l|}{}                         &                  & $4$      & $7$     &       \\ \cline{2-5}
\multicolumn{1}{|l|}{}                         & $d = 3$          &  $2$   &  $4$   &  \ref{sec:subsec-q=3-geometric-d=3-case} \\ \cline{2-5}
\multicolumn{1}{|l|}{}                         & $d > 3$          &  -   &  -   &  \ref{sec:subsec-q=3-geometric-d=4-case}
    \\ \hline
\end{tabular}
\caption{Bounds for the relative class number 2 problem.}
\label{table:bounds table}
\end{table}
\begin{table}
    \setlength{\arrayrulewidth}{0.2mm} 
    \setlength{\tabcolsep}{5pt}
    \renewcommand{\arraystretch}{1.2}
    \centering
    \begin{tabular}{|c|c|c|}
    \hline
    \rowcolor{headercolor} 
    $q$ & Isogeny class of $J'$ & Isogeny class of $C$ \\ 
     \hline 
    \multirow{7}{*}{5}& \avlink{2.5.ag\_s} & \avlink{1.5.ac} \\
    \cline{2-3}
    & \avlink{2.5.af\_o} & \avlink{1.5.ab} \\
    \cline{2-3}
    & \avlink{2.5.ae\_k} & \avlink{1.5.a} \\
    \cline{2-3}
    & \avlink{2.5.ad\_g} & \avlink{1.5.b} \\
    \cline{2-3}
    & \avlink{2.5.ac\_c} & \avlink{1.5.c} \\
    \cline{2-3}
    & \avlink{2.5.ab\_ac} & \avlink{1.5.d} \\
    \cline{2-3}
    & \avlink{2.5.a\_ag} & \avlink{1.5.e} \\
    \hline
    \multirow{5}{*}{4} & \avlink{2.4.ae\_l} & \avlink{1.4.ab} \\
    \cline{2-3}
    & \avlink{2.4.ad\_i} & \avlink{1.4.a} \\
    \cline{2-3}
    & \avlink{2.4.ac\_f} & \avlink{1.4.b} \\
    \cline{2-3}
    & \avlink{2.4.ab\_c} & \avlink{1.4.c} \\
    \cline{2-3}
    & \avlink{2.4.a\_ab} & \avlink{1.4.d} \\
    \hline
    \multirow{13}{*}{3} & \avlink{2.3.ae_k} & \avlink{1.3.ac} \\
    \cline{2-3}
     & \avlink{2.3.ac_g} & \avlink{1.3.a} \\
    \cline{2-3}
     & \avlink{2.3.ab_e} & \avlink{1.3.b} \\
    \cline{2-3}
     & \avlink{2.3.a_c} & \avlink{1.3.c} \\
    \cline{2-3}
     & \avlink{2.3.b_a} & \avlink{1.3.d} \\
    \cline{2-3}
     & \avlink{3.3.ae_k_as} & \avlink{1.3.b} \\
    \cline{2-3}
     & \avlink{3.3.ad_f_ag} & \avlink{1.3.c} \\
    \cline{2-3}
     & \avlink{3.3.ac_a_g} & \avlink{1.3.d} \\
    \cline{2-3}
     & \avlink{3.3.ae_l_ay} & \avlink{1.3.a} \\
    \cline{2-3}
     & \avlink{3.3.ad_h_aq} & \avlink{1.3.b} \\
    \cline{2-3}
     & \avlink{3.3.ac_d_ai} & \avlink{1.3.c} \\
    \cline{2-3}
     & \avlink{3.3.ab_ab_a} & \avlink{1.3.d} \\
    \cline{2-3}
     & \avlink{4.3.ae_f_a_ag} & \avlink{1.3.d} \\
     \hline
\end{tabular}

  \caption{Pairs of isogeny classes for $(J', C)$ corresponding to maps $C'\to C$ of curves over $\FF_q$ of relative class number two, when $C$ is an elliptic curve.}
    \label{tab:geometric-q-C-elliptic}
\end{table}

\section{\texorpdfstring{$m = 2$}{m=2} and \texorpdfstring{$q > 4$}{q>4}}
\label{sec:q=5}

We prove Proposition~\ref{prop:summary} for $q > 4$.
\subsection{Analysis of the Prym}

From the Weil bound, if $\#A(\FF_q) = 2$, then $q=5$ and $\dim(A)=1$. There is a unique isogeny class satisfying these requirements, namely \avlink{1.5.ae}.

\subsection{Constant case} 
\label{sec:subsec-q=5-constant-case}
In this case, by Corollary~\ref{cor:constan-composite-deg}, $d$ must be prime; by 
Corollary~\ref{cor:constan-prime-deg}, we must in fact have $d=2$.
Consequently,
$A$ must be the quadratic twist of $J = C$. This implies that $P_C(T) = T^2+2T+5$, and $C$ is isogenous to the elliptic curve with label \avlink{1.5.e}.

\subsection{Purely geometric case}
\label{sec:subsec-q=5-geometric-case}
In this case, Riemann--Hurwitz gives the inequality $$(d-1)(g-1) \leq 1,$$ implying that either $g\leq 1$ or $g=d=2$.
We separate into cases based on $g$.

\subsubsection{$g = 0$} In this case, we have that the Prym variety $A$ is isogenous to $J'$. Since $h_{F'/F} = h_{F'} = \#A(\FF_5) = \#J'(\FF_5)$, the analysis of \Cref{sec:subsec-q=5-constant-case} gives that $J' = C'$ must be an elliptic curve in the isogeny class \avlink{1.5.ae}.

\subsubsection{$g = 1$} In this case, $J' \sim J\times A = C\times A$. In particular, $A$ is an elliptic curve in the isogeny class \avlink{1.5.ae}, and for every Jacobian of the form $J' = C\times A$, we obtain a covering $C' \to C$ by composing the Abel--Jacobi map with the natural projection. To find these coverings, we loop through all isogeny classes of elliptic curves $C$ over $\FF_5$ and save the ones for which $C\times A$ is a Jacobian. This \href{https://www.lmfdb.org/Variety/Abelian/Fq/?q=5&jacobian=yes&simple_factors=1.5.ae&dim1_factors=2}{LMFDB query} confirms the results of Table~\ref{tab:geometric-q-C-elliptic}.
%\begin{table}[ht]
%    \setlength{\arrayrulewidth}{0.2mm} 
%    \setlength{\tabcolsep}{5pt}
%    \renewcommand{\arraystretch}{1.2}
%    \centering
%    \begin{tabular}{|c|c|}
%    \hline
%    \rowcolor{headercolor} 
%    Isogeny class of $J'$ & Isogeny class of $C$ \\ 
%     \hline 
%     \avlink{2.5.ag\_s} & \avlink{1.5.ac} \\
%    \hline
%     \avlink{2.5.af\_o} & \avlink{1.5.ab} \\
%    \hline
%     \avlink{2.5.ae\_k} & \avlink{1.5.a} \\
%    \hline
%     \avlink{2.5.ad\_g} & \avlink{1.5.b} \\
%    \hline
%     \avlink{2.5.ac\_c} & \avlink{1.5.c} \\
%    \hline
%     \avlink{2.5.ab\_ac} & \avlink{1.5.d} \\
%    \hline
%     \avlink{2.5.a\_ag} & \avlink{1.5.e} \\
%    \hline
%\end{tabular}
%    \caption{Pairs of isogeny classes for $(J', C)$ corresponding to maps $C'\to C$ of curves over $\FF_5$ of relative class number two, when $C$ is an elliptic curve.}
%    \label{tab:constant-q=5-C-elliptic}
%\end{table}

\subsubsection{$g = 2$} 

In this case, $d = 2$, $g' = 3$, and there is nothing more to check now.

%%%%%%%%%%%%%%%%%%%%%%%%%%%%%%%%%%

\section{\texorpdfstring{$m=2$}{m=2} and \texorpdfstring{$q = 4$}{q=4}}
\label{sec:q=4}

We prove Proposition~\ref{prop:summary} for $q = 4$.

\subsection{Analysis of the Prym}

Up to isogeny, we can write
\[
A \sim E^{\dim(A)-\dim(B)} \times B,
\]
where $E$ is the elliptic curve with $\#E(\FF_4) = 1$ (isogeny class \avlink{1.4.ae}) and $B$ is 
an abelian variety not containing $E$ as a factor and satisfying $\#B(\FF_4) = 2$. By \cite[Theorem~3.2]{Kadets2021}, we have $\#B(\FF_4) \geq 2^{\dim(B)}$ and so $\dim(B) = 1$;
specifically, $B$ is an elliptic curve with $\#B(\FF_4) = 2$ (isogeny class \avlink{1.4.ad}).
We have
\begin{gather*}
T_{E,4} = 4, \qquad T_{B,4} = 3, \\
T_{E,4^2} = 8, \qquad T_{B, 4^2} = 1,
\end{gather*}
and so
\begin{equation} \label{eq:trace of Prym over F3}
T_{A,4} = 4 \dim(A) - 1, \qquad T_{A,4^2} = 8 \dim(A) - 7.    
\end{equation}

\subsection{Constant case}
\label{sec:subsec-q=4-constant-case}
By Corollary~\ref{cor:constan-composite-deg}, the only way for $F'/F$ to be constant of composite degree would be to have an intermediate field $F''/F$ with $[F''/F] = 1$; but in this case, we would have $q'' \geq 4^2 > 5$. Consequently, if $F'/F$ is constant, then $d=2$.

We have $T_{A,4}=-T_{C,4} = \#C(\FF_4) -q - 1$ with $\dim(A) = g$.
By \Cref{eq:trace of Prym over F3},
$$\#C(\FF_4)  = 5 + 4\dim(A) -1 = 4 + 4g.$$

Using Lemma~\ref{lemma:points-bounds} we get $4+4g = \#C(\FF_4) \leq 2.31g + 7.31$ from which we can bound $g\leq 1$, so $g = 1$. Since $A$ is the quadratic twist of $C$, this implies that $P_C(T) = T^2+3T+4$, and $C$ is isogenous to the elliptic curve with label \avlink{1.4.d}.

%Checking in \url{manypoints.org} for the maximum amount of points that a curve of genus $g\leq 6$ over $\FF_4$ can have, we have the following
%\begin{align*}
%g=2 & \Longrightarrow \#C(\FF_{4}) \leq 10 < 2\cdot2 + 4\\
%g=3 & \Longrightarrow \#C(\FF_{4}) \leq 14 < 2\cdot3 + 4\\
%g=4 & \Longrightarrow \#C(\FF_{4}) \leq 15 < 2\cdot4 + 4\\
%g=5 & \Longrightarrow \#C(\FF_{4}) \leq 17 < 2\cdot5 + 4\\
%g=6 & \Longrightarrow \#C(\FF_{4}) \leq 20 < 2\cdot6 + 4
%\end{align*}
%so we deduce that the only possibility is $g = 1$. Since $A$ is the quadratic twist of $C$, %this implies that $P_C(T) = T^2+3T+4$, and $C$ is isogenous to the elliptic curve with label \avlink{1.4.d}.

\subsection{Purely geometric case}

By \Cref{eq:trace-less-than-points}, \Cref{eq:trace of Prym over F3}, and Lemma~\ref{lemma:points-bounds}, 
\begin{equation} \label{eq:bounds F4 geometric}
2.31g + 7.31 \geq \#C(\FF_4) \geq T_{A,4} = 4(g'-g-1)+3. 
\end{equation}

\subsubsection{$g =0$}
\label{sec:subsec-q=4-geometric-g=0-case}

In this case, \Cref{eq:bounds F4 geometric} gives $4g'-1 \leq 5 = \#C(\FF_4)$, so that $g' = 1$. Therefore, $C'$ is isogenous to $B$ itself.

\subsubsection{$g = 1$} 
\label{sec:subsec-q=4-geometric-g=1-case}

When $g=1$, we get $9 \geq 4g' - 5$ and so $g' \leq 3$.

For $g' = 2$, we have $J' \sim J \times B$; according to  this \href{https://www.lmfdb.org/Variety/Abelian/Fq/?q=4&g=2&jacobian=not_no\&simple_factors=1.4.ad}{LMFDB query}, the possible isogeny classes for $J'$ are \avlink{2.4.ae\_l}, \avlink{2.4.ad\_i}, \avlink{2.4.ac\_f}, \avlink{2.4.ab\_c}, \avlink{2.4.a\_ab}.

For $g' = 3$, we have $J' \sim J \times E \times B$.
This \href{https://www.lmfdb.org/Variety/Abelian/Fq/?q=4&g=3&jacobian=not_no&simple_factors=1.4.ad%2C1.4.ae}{LMFDB query} shows 
that there are no Jacobians of this form.

This confirms the values listed in Table~\ref{tab:geometric-q-C-elliptic}.

%\begin{table}[ht]
%\setlength{\arrayrulewidth}{0.2mm} 
%    \setlength{\tabcolsep}{5pt}
%    \renewcommand{\arraystretch}{1.2}
%    \centering
%    \begin{tabular}{|c|c|}
%    \hline
%    \rowcolor{headercolor}
%    Isogeny class of $J'$ & Isogeny class of $C$ \\ 
%    \hline 
%    \avlink{2.4.ae\_l} & \avlink{1.4.ab} \\
%    \hline
%    \avlink{2.4.ad\_i} & \avlink{1.4.a} \\
%    \hline
%    \avlink{2.4.ac\_f} & \avlink{1.4.b} \\
%    \hline
%    \avlink{2.4.ab\_c} & \avlink{1.4.c} \\
%    \hline
%    \avlink{2.4.a\_ab} & \avlink{1.4.d} \\
%    \hline
%    \end{tabular}
%    \caption{Pairs of isogeny classes for $(J', C)$ corresponding to maps $C'\to C$ of curves over $\FF_4$ of relative class number two, when $C$ is an elliptic curve.}
%    \label{tab:pure-geom-q=4-C-elliptic}
%\end{table}

\subsubsection{$g \geq 2$ and $d = 2$}
\label{sec:subsec-q=4-geometric-d=2-case}

In this case,
\[
2.31g + 7.31 \geq \#C(\FF_4) \geq T_{A,4} \geq 4(g-1) - 1 = 4g-5,
\]
and so $g \leq 7$. In addition, 
by Lemma~\ref{lemma:points-bounds} and \Cref{eq:trace-upper-bound-deg-2},
\begin{align*}
\#C(\FF_{4^2}) &\geq 2\#C(\FF_4) + T_{A,4^2} - t \geq 2\#C(\FF_4) + 8 \dim(A) - 7 - t
\geq 2\#C(\FF_4) + 8g - 15, \\
&\geq 2T_{A,4} + 8g-15 = 8 \dim(A) - 2 + 8g-15 \geq 16g - 25.
\end{align*}
From Table~\ref{table:many-points}, we see that for $g=7$, $\#C(\FF_{4^2}) \leq 69 < 87 = 16g-25$; for $g=6$, $\#C(\FF_{4^2}) \leq 65 < 71 = 16g-25$; and for $g=5$, $\#C(\FF_{4^2}) \leq 53 < 55 = 16g-25$.

For $g=4$, we have
\[
45 \geq \#C(\FF_{4^2}) \geq 16g-25 + 15(g'-2g+1) = 39 + 15(g'-2g+1),
\]
and so $g' = 7$ and $F'/F$ is unramified. 
We may rule out this case as follows. 
Since $F'/F$ is unramified, 
$\#C'(\FF_4) \equiv 0 \pmod{2}$ and so $\#C(\FF_4) - T_{A,4} = \#C(\FF_4)-11$ is a nonnegative even integer. Since  $45 \geq \#C(\FF_{4^2}) \geq 2\#C(\FF_4) + 17$ by Table~\ref{table:many-points}
and $\#C(\FF_{4^2}) \equiv \#C(\FF_4) \pmod{2}$, the pair $(\#C(\FF_4), \#C(\FF_{4^2}))$
must be one of
\[
\href{https://www.lmfdb.org/Variety/Abelian/Fq/?q=4&g=4&jacobian=not_no&curve_point_count=%5B11%2C+39%5D}{(11, 39)}, \href{https://www.lmfdb.org/Variety/Abelian/Fq/?q=4&g=4&jacobian=not_no&curve_point_count=%5B11%2C+41%5D}{(11, 41)}, \href{https://www.lmfdb.org/Variety/Abelian/Fq/?q=4&g=4&jacobian=not_no&curve_point_count=%5B11%2C+43%5D}{(11, 43)}, \href{https://www.lmfdb.org/Variety/Abelian/Fq/?q=4&g=4&jacobian=not_no&curve_point_count=%5B11%2C+45%5D}{(11, 45)}, %\href{https://www.lmfdb.org/Variety/Abelian/Fq/?q=4&g=4&jacobian=not_no&curve_point_count=%5B12%2C+42%5D}{(12, 42)}, 
%\href{https://www.lmfdb.org/Variety/Abelian/Fq/?q=4&g=4&jacobian=not_no&curve_point_count=%5B12%2C+44%5D}{(12, 44)}, 
\href{https://www.lmfdb.org/Variety/Abelian/Fq/?q=4&g=4&jacobian=not_no&curve_point_count=%5B13%2C+43%5D}{(13, 43)}, 
\href{https://www.lmfdb.org/Variety/Abelian/Fq/?q=4&g=4&jacobian=not_no&curve_point_count=%5B13%2C+45%5D}{(13, 45)}.
\]
The indicated LMFDB searches show that none of these cases can occur.
  
For $g=3$, we have $14 \geq \#C(\FF_4) \geq 4(g'-g)-1$.
For the exhaust, we record the constraint
\begin{equation} \label{eq:constraint-q4-g3-gg5}
(q,g,g') = (4,3,5) \Longrightarrow \#C(\FF_4) \geq 7, \#C(\FF_{4^2}) \geq 2\#C(\FF_4)+9.
\end{equation}
If $g' = 6$, this improves to $\#C(\FF_4) \geq 11$ and
$\#C(\FF_{4^2}) \geq 38$. Since $38 \geq \#C(\FF_{4^2})$ by Table~\ref{table:many-points}
and $\#C(\FF_{4^2}) \equiv \#C(\FF_4) \pmod{2}$, the pair $(\#C(\FF_4), \#C(\FF_{4^2}))$ must be one of
\[
\href{https://www.lmfdb.org/Variety/Abelian/Fq/?q=4&g=3&jacobian=not_no&curve_point_count=%5B12%2C+38%5D}{(12, 38)}, \href{https://www.lmfdb.org/Variety/Abelian/Fq/?q=4&g=3&jacobian=not_no&curve_point_count=%5B14%2C+38%5D}{(14, 38)}.
\]
The indicated LMFDB searches show that none of these cases can occur.

%and so $g' \leq 6$. 
%If $g' = 6$, then $\#C(\FF_4) \geq 7$ and $\#C(\FF_{4^2}) = 38$; since $\#C(\FF_4) \equiv \#C(\FF_{4^2}) \pmod{2}$ we also have $\#C(\FF_4) \geq 8$.
%We next enumerate Weil polynomials satisfying all of the following conditions:
%\begin{itemize}
%    \item The number of places of degree $i$ is nonnegative.
%    \item The Serre--Howe--Lauter resultant condition holds.
%    \item We have $\#C(\FF_4) \geq 8$.
%    \item We have $\#C(\FF_{4^2}) \geq 38$.
%\end{itemize}
%This yields a unique isogeny class for $J$: \avlink{3.4.d\_p\_z}, which contains the Jacobians of two curves.
%Using Magma to perform explicit class field theory, we look for quadratic extensions of these two function fields with relative class number 2; we find no cases.
%\santi{Do we have data for non-hyperelliptic $g=3$ over $\FF_4$?}
%\kiran{This data is in LMFDB.} \tdi{Loop through $q = 4$, $g=3$, and $d=2$.} 

For $g=2$, we have $10 \geq \#C(\FF_4) \geq 4(g'-g) - 1$ and so $g' \leq 4$. 
For the exhaust, we record the constraint
\begin{equation} \label{eq:constraint-q4-g2-gg4}
(q,g,g') = (4,2,4) \Longrightarrow \#C(\FF_4) \geq 7.
\end{equation}
%Using Magma to perform explicit class field theory, we may loop over all curves of genus 2 to look for quadratic extensions with relative class number 2. All these quadratic extensions have $g' = 3$. \kiran{Calculation in notebook "q4\_geometric.ipynb" when $g'=3$; still need to do the case $g'=4$.} \santi{It seems that all such quadratic extensions have $g'=3$. The calculation is in \texttt{q4\_geometric.ipynb}.}

\subsubsection{$g \geq 2$ and $d > 2$}
\label{sec:subsec-q=4-geometric-d=3-case}

In this case, by Riemann--Hurwitz, \Cref{eq:trace of Prym over F3}, and Lemma~\ref{lemma:points-bounds}, 
\[
2.31g + 7.31 \geq \#C(\FF_4) \geq T_{A,4} \geq 4(d-1)(g-1) - 1 \geq 8(g-1) - 1,
\]
which implies $g = 2$. Now from Table~\ref{table:many-points} we have
\[
10 \geq \#C(\FF_4) \geq T_{A,4} = 4 \dim(A) - 1,
\]
and so $\dim(A) \leq 2$; consequently, $d=3$ and the extension is unramified.

By Lemma~\ref{lem:prime-deg-congruence-geometric}, $F'/F$ cannot be cyclic.
As in section \ref{subsec:noncyclic-description},
the quadratic resolvent of the Galois closure $F''/F$ must be purely geometric.
The Prym of the quadratic resolvent is an elliptic curve $E'$ such that
\[
0 \leq \#C''(\FF_4) = \#C(\FF_4) - 2T_{A,4} - T_{E',4} = \#C(\FF_4) - 14 - T_{E',4} \leq \#C(\FF_4) - 10;
\]
consequently, $C$ must be the unique curve with $\#C(\FF_4) = 10$ and $J$ belongs to the isogeny class \avlink{2.4.f\_n}. However, in this case $\#J(\FF_4) = 55$ is odd, and so by class field theory (as in \cite[(7.5)]{Kedlaya2022}) $F$ cannot admit a nonconstant unramified quadratic extension, a contradiction.

%Using Magma to perform explicit class field theory, we may loop over all curves of genus 2 to look for cyclic cubic extensions with relative class number 2. We find no such coverings.
%\kiran{Calculation in notebook "q=4\_geometric.ipynb"}
%\tdi{Print tables.}

%If the quadratic resolvent is constant, then the quadratic twist $\tilde{F}$ of $F$ admits a \emph{cyclic} cubic unramified extension $\tilde{F}'$ whose Prym variety is the quadratic twist $\tilde{A}$ of $A$.
%We may search for such coverings by looping over all curves of genus 2 to look for cyclic cubic extensions for which %$P_A(-1) = 2$. We find no such coverings.
%\kiran{Calculation in notebook "q=4\_geometric.ipynb"}
%\tdi{Print tables.}

%%%%%%%%%%%%%%%%%%%%%%%%%%%%%%%%%%

\section{\texorpdfstring{$m=2$}{m=2} and \texorpdfstring{$q = 3$}{q=3}}
\label{sec:q=3}

We prove Proposition~\ref{prop:summary} for $q = 3$.

\subsection{Analysis of the Prym}

As in the case $q=4$, up to isogeny we can write
\[
A \sim E^{\dim(A)-\dim(B)} \times B,
\]
where $E$ is an elliptic curve with $\#E(\FF_3) = 1$ in the isogeny class \avlink{1.3.ad} and $B$ is an abelian variety not containing $E$ as a factor and satisfying $\#B(\FF_3) = 2$. By \Cref{lemma:exp-lower-bounds}, we have $\#B(\FF_3) \geq 1.359^{\dim(B)}$ and so $\dim(B) \leq 2$. By this \href{https://www.lmfdb.org/Variety/Abelian/Fq/?q=3&simple=yes&abvar_point_count=%5B2%5D}{LMFDB query}, we obtain two possibilities for the isogeny class of $B$: an elliptic curve $B_1$ in \avlink{1.3.ac} or a simple abelian surface $B_2$ in \avlink{2.3.ae\_i}. We have
\begin{gather*}
T_{E,3} = 3, \qquad T_{B_1,3} = 2, \qquad T_{B_2, 3} = 4, \\
T_{E,3^2} = 3, \qquad T_{B_1,3^2} = -2, \qquad T_{B_2, 3^2} = 0,
\end{gather*}
and therefore
\begin{equation}
\label{eq:prym-trace-q3}
(T_{A,3}, T_{A,3^2}) = 
\begin{cases} 
(3\dim(A)-1, 3\dim(A)-5) & B = B_1, \\
(3\dim(A)-2, 3\dim(A)-6) & B = B_2.
\end{cases}
\end{equation}

\subsection{Constant case}
\label{sec:subsec-q=3-constant-case}
As in the case $q=4$, the fact that $3^2 > 5$ means that if $F'/F$ is constant, we must have $d=2$.

When $A\sim E^{\dim(A) - \dim(B_1)}\times B_1$, by Lemma \ref{lemma:points-bounds} we get 
\[3g + 3 = \#C(\FF_3) \leq 1.91g + 5.52,\]
and we deduce that $g\leq 2$.

When $A\sim E^{\dim(A) - \dim(B_2)}\times B_2$, by Lemma \ref{lemma:points-bounds} we get 
\[3g + 2 = \#C(\FF_3) \leq 1.91g + 5.52,\]
and we deduce that $g\leq 3$. However, Table~\ref{table:many-points} indicates that for $g=3$, $\#C(\FF_3) \leq 10 < 11 = 3g+2$; we thus again obtain $g \leq 2$.

If $g=1$, then $A \sim E$ and so $J$ must belong to the isogeny class \avlink{1.3.d}.
If $g=2$, then $A$ cannot be isogenous to $E \times B_1$, as then its quadratic twist would belong to the isogeny class \avlink{2.3.f_m} which does not contain a Jacobian; hence $A \sim B_2$
and $J$ belongs to \avlink{2.3.e_i}.

\subsection{Purely geometric case} 

\subsubsection{$g =0$}
\label{sec:subsec-q=3-geometric-g=0-case}
In this case we have that $\#C(\FF_3)= 4$, so when  $B \sim  B_1
$, by using \Cref{eq:prym-trace-q3} and Lemma~\ref{trace-bound}, we have $$3(g'-g)-1=3\dim A-1=T_{A,3}\leq \#C(\FF_3)=4.$$ Therefore, $g'= 1$ and $C'$ is isogenous to $B_1$ itself.

In the case $B\sim  B_2
$,  $$3(g'-g)-2=3\dim A-2=T_{A,3}\leq \#C(\FF_3)\leq 4.$$ Therefore, $g'= 2$ and $C'$ is isogenous to $B_2$ itself.

\subsubsection{$g = 1$} 
\label{sec:subsec-q=3-geometric-g=1-case}
For $g=1$, $\#C(\FF_3)\leq 7$, so in the first case ($B\sim B_1$) we have $$3(g'-g)-2\leq 7$$ from what we get $g'\leq 3$. 

\begin{itemize}
    \item For $g' = 2$, we have $J' \sim J \times B$; according to  this \href{https://www.lmfdb.org/Variety/Abelian/Fq/?q=3&g=2&jacobian=yes&simple_factors=1.3.ac}{LMFDB query}, the possible isogeny classes for $J'$ are \avlink{2.3.ae\_k}, \avlink{2.3.ac\_g}, \avlink{2.3.ab\_e}, \avlink{2.3.a\_c}, \avlink{2.3.b\_a}.
    \item For $g' = 3$, we have $J' \sim J \times E \times B$ and according to this \href{https://www.lmfdb.org/Variety/Abelian/Fq/?q=3&g=3&jacobian=yes&simple_factors=1.3.ac%2C1.3.add}{LMFDB query}, the possible isogeny clases are \avlink{3.3.ae\_k\_as}, \avlink{3.3.ad\_f\_ag}, \avlink{3.3.ac\_a\_g}.

\end{itemize}

In the second case ($B\sim B_2$), we have $g'\leq 4$. 
\begin{itemize}
    \item For $g' = 3$, we have $J' \sim J \times B$; according to  this \href{https://www.lmfdb.org/Variety/Abelian/Fq/?q=3&g=3&jacobian=yes&simple_factors=2.3.ae_i}{LMFDB query}, the possible isogeny classes for $J'$ are \avlink{3.3.ae\_l\_ay}, \avlink{3.3.ad\_h\_aq}, \avlink{3.3.ac\_d\_ai}, \avlink{3.3.ab\_ab\_a}.
    \item For $g' = 4$, we have $J' \sim J \times E \times B$ and according to this \href{https://www.lmfdb.org/Variety/Abelian/Fq/?q=3&g=4&jacobian=yes&simple_factors=2.3.ae_i%2C1.3.ad}{LMFDB query}, the only possible isogeny class is \avlink{4.3.ae\_f\_a\_ag}.
\end{itemize}

This confirms the values listed in Table~\ref{tab:geometric-q-C-elliptic}.

%%%%%%%%%%%%%%%%%%%%%%%%%%%%%%%%%%

\subsubsection{$g \geq 2$ and $d = 2$}
\label{sec:subsec-q=3-geometric-d=2-case}

By \Cref{eq:trace-upper-bound-deg-2},
\[
1.91g + 5.52 \geq \#C(\FF_3) \geq T_{A,3} \geq 3 \dim(A) - 2 \geq 3g - 5,
\]
and so $g \leq 9$. By Lemma~\ref{lemma:points-bounds},
\begin{align*}
\#C(\FF_{3^2}) &\geq 2\#C(\FF_3) + T_{A,3^2} - t \geq 2 \#C(\FF_3) + 3 \dim(A) - 6 - t \geq 2\#C(\FF_3) + 3g-9, \\
&\geq 2T_{A,3} + 3g-9 = 6 \dim(A) - 4 + 3g-9
\geq 9g-19.
\end{align*}
From Table~\ref{table:many-points}, we see that for $g=9$, $\#C(\FF_{3^2}) \leq 50 < 62 = 9g-19$;
for $g=8$, $\#C(\FF_{3^2}) \leq 46<  53 = 9g-19$; and for $g=7$,
$\#C(\FF_{3^2}) \leq 43 < 44 = 9g-19$. We deduce that $g \leq 6$.

For $g=6$, we have $14 \geq \#C(\FF_3) \geq 3(g'-g) - 2$, so $g' = 11$ and $F'/F$ is unramified. We may rule out this case as follows. Since
$\#C(\FF_3) \geq 13$, $38 \geq \#C(\FF_{3^2}) \geq 2\#C(\FF_3) + 9$,
and $\#C(\FF_{3^2}) \equiv \#C(\FF_3) \pmod{2}$, the pair $(\#C(\FF_3), \#C(\FF_{3^2}))$ must be one of
\[
\href{https://www.lmfdb.org/Variety/Abelian/Fq/?q=3&g=4&jacobian=not_no&curve_point_count=%5B13%2C35%5D}{(13, 35)}, 
\href{https://www.lmfdb.org/Variety/Abelian/Fq/?q=3&g=4&jacobian=not_no&curve_point_count=%5B13%2C37%5D}{(13, 37)}, 
\href{https://www.lmfdb.org/Variety/Abelian/Fq/?q=3&g=4&jacobian=not_no&curve_point_count=%5B13%2C35%5D}{(14, 38)}.
\]
The indicated LMFDB searches show that none of these cases can occur.

For $g=5$, we have $\#C(\FF_3) \geq 3g-5 = 10$
and $13 \geq \#C(\FF_3) \geq 3(g'-g) - 2$; the latter yields $g' \leq 10$.
We rule out the case $g'=10$ as follows. Since $\#C(\FF_3) \geq 10$, 
$35 \geq \#C(\FF_{3^2}) \geq 2\#C(\FF_3) + 8$, and $\#C(\FF_{3^2}) \equiv \#C(\FF_3) \pmod{2}$, the pair  $(\#C(\FF_3), \#C(\FF_{3^2}))$ must be one of
\[
\href{https://www.lmfdb.org/Variety/Abelian/Fq/?q=3&g=5&jacobian=not_no&curve_point_count=%5B10%2C28%5D}{(10, 28)}, 
\href{https://www.lmfdb.org/Variety/Abelian/Fq/?q=3&g=5&jacobian=not_no&curve_point_count=%5B10%2C30%5D}{(10, 30)}, 
\href{https://www.lmfdb.org/Variety/Abelian/Fq/?q=3&g=5&jacobian=not_no&curve_point_count=%5B10%2C32%5D}{(10, 32)}, 
\href{https://www.lmfdb.org/Variety/Abelian/Fq/?q=3&g=5&jacobian=not_no&curve_point_count=%5B10%2C34%5D}{(10, 34)},
\href{https://www.lmfdb.org/Variety/Abelian/Fq/?q=3&g=5&jacobian=not_no&curve_point_count=%5B11%2C31%5D}{(11, 31)}, 
\href{https://www.lmfdb.org/Variety/Abelian/Fq/?q=3&g=5&jacobian=not_no&curve_point_count=%5B11%2C33%5D}{(11, 33)}, 
\href{https://www.lmfdb.org/Variety/Abelian/Fq/?q=3&g=5&jacobian=not_no&curve_point_count=%5B11%2C35%5D}{(11, 35)},
\href{https://www.lmfdb.org/Variety/Abelian/Fq/?q=3&g=5&jacobian=not_no&curve_point_count=%5B12%2C32%5D}{(12, 32)}, 
\href{https://www.lmfdb.org/Variety/Abelian/Fq/?q=3&g=5&jacobian=not_no&curve_point_count=%5B12%2C34%5D}{(12, 34)},
\href{https://www.lmfdb.org/Variety/Abelian/Fq/?q=3&g=5&jacobian=not_no&curve_point_count=%5B13%2C35%5D}{(13, 35)}.
\]
The indicated LMFDB searches show that none of these cases can occur.

For $g=5$, it still remains to rule out $g' = 9$ as follows.
Since $\#C(\FF_3) \geq 10$,
$35 \geq \#C(\FF_{3^2}) \geq 2\#C(\FF_3) + 6$, $\#C(\FF_{3^2}) \equiv \#C(\FF_3) \pmod{2}$,  and we have already excluded some options for the pair
$(\#C(\FF_3), \#C(\FF_{3^2}))$, this pair must be one of
\[
\href{https://www.lmfdb.org/Variety/Abelian/Fq/?q=3&g=5&jacobian=not_no&curve_point_count=%5B10%2C26%5D}{(10, 26)}, 
\href{https://www.lmfdb.org/Variety/Abelian/Fq/?q=3&g=5&jacobian=not_no&curve_point_count=%5B11%2C29%5D}{(11, 29)}, 
\href{https://www.lmfdb.org/Variety/Abelian/Fq/?q=3&g=5&jacobian=not_no&curve_point_count=%5B12%2C30%5D}{(12, 30)}, 
\href{(https://www.lmfdb.org/Variety/Abelian/Fq/?q=3&g=5&jacobian=not_no&curve_point_count=%5B13%2C33%5D}{(13, 33)},
\href{https://www.lmfdb.org/Variety/Abelian/Fq/?q=3&g=5&jacobian=not_no&curve_point_count=%5B14%2C34%5D}{(14, 34)}.
\]
The indicated LMFDB searches show that none of these cases can occur except possibly the first one, for which we have the candidate isogeny classes
\avlink{5.3.g\_ba\_da\_hh\_nw} and \avlink{5.3.g\_ba\_db\_hh\_ob} for $J$.
The second one has the property that $\#J(\FF_3)$ is odd, so a curve with Jacobian $J$ cannot admit an \'etale double cover.
Using the following lemma, we deduce that the case $g=5$ cannot occur.

\begin{lemma}
    The isogeny class \avlink{5.3.g\_ba\_da\_hh\_nw} does not contain the isogeny class of a Jacobian.
\end{lemma}
\begin{proof}
(Suggested by Everett Howe)
Suppose the contrary that the isogeny class contains the Jacobian of a curve $C$.
The ``resultant 2'' criterion \cite[Theorem~4.5]{Howe2023} shows that $C$ must be a double cover of 
a curve whose Jacobian belongs to one of the isogeny classes \avlink{2.3.d\_i} or \avlink{3.3.d\_j\_s};
however, the former does not contain a Jacobian (e.g., by the ``resultant 1'' criterion).
The latter contains the Jacobian of a unique curve $C_0$; the double cover $C \to C_0$ must be \'etale. We may now run over \'etale double covers of $C_0$ to see that none of them have the correct number of $\FF_3$-rational points; this computation can be found in the notebook \href{https://github.com/sarangop1728/twice-class-number/blob/main/q3_geometric.ipynb}{\texttt{q3_geometric.ipynb}} listed in \Cref{tab:meta-table}.
\end{proof}

For $g = 4$, Table~\ref{tab:datos} gives $12 \geq \#C(\FF_3) \geq 3(g'-g) - 2$, so $g'-g \leq 4$ and $g' \leq 8$.
For the exhaust, we record the constraint
\begin{equation} \label{eq:constraint-q3-g4-gg7}
(q,g,g') = (3,4,7) \Longrightarrow \#C(\FF_4) \geq 3, \#C(\FF_{3^2}) \geq 17.
\end{equation}
We rule out $g' = 8$ as follows: we would have
$12 \geq \#C(\FF_3) \geq 3\cdot 4-2 = 10$
and $30 \geq \#C(\FF_{3^2}) \geq 2\#C(\FF_3) + 3\cdot 4 - 6 - 2 
\geq 24$.
The pair $(\#C(\FF_3), \#C(\FF_{3^2}))$ must be one of
\[
\href{https://www.lmfdb.org/Variety/Abelian/Fq/?q=3&g=4&jacobian=not_no&curve_point_count=%5B10%2C24%5D}{(10, 24)}, 
\href{https://www.lmfdb.org/Variety/Abelian/Fq/?q=3&g=4&jacobian=not_no&curve_point_count=%5B10%2C26%5D}{(10, 26)}, 
\href{https://www.lmfdb.org/Variety/Abelian/Fq/?q=3&g=4&jacobian=not_no&curve_point_count=%5B10%2C28%5D}{(10, 28)}, 
\href{(https://www.lmfdb.org/Variety/Abelian/Fq/?q=3&g=4&jacobian=not_no&curve_point_count=%5B10%2C30%5D}{(10, 30)}, 
\href{https://www.lmfdb.org/Variety/Abelian/Fq/?q=3&g=4&jacobian=not_no&curve_point_count=%5B11%2C27%5D}{(11, 27)}, 
\href{https://www.lmfdb.org/Variety/Abelian/Fq/?q=3&g=4&jacobian=not_no&curve_point_count=%5B11%2C29%5D}{(11, 29)}, 
\href{https://www.lmfdb.org/Variety/Abelian/Fq/?q=3&g=4&jacobian=not_no&curve_point_count=%5B12%2C28%5D}{(12, 28)}, 
\href{https://www.lmfdb.org/Variety/Abelian/Fq/?q=3&g=4&jacobian=not_no&curve_point_count=%5B12%2C28%5D}{(12, 30)}.
\]
The indicated LMFDB searches show that none of these cases can occur.

%\tdi{Run cyclic covers for $q=3, g=4$, and $g' \leq 8$.}

For $g=3$, Table~\ref{tab:datos} gives $10 \geq \#C(\FF_3) \geq 3(g'-g) -2$, so $g'-g \leq 4$ and $g' \leq 7$.
For the exhaust, we record the constraint
\begin{equation} \label{eq:constraint-q3-g3-gg6}
(q,g,g') = (3,3,6) \Longrightarrow \#C(\FF_3) \geq 7, \#C(\FF_{3^2}) \geq 2\#C(\FF_{3^2})+1.
\end{equation}
We rule out $g'=7$ as follows: we would have $10 \geq \#C(\FF_3) \geq 10$
and $\#C(\FF_{3^2}) \geq 2\#C(\FF_3) + 3(g'-g') -6-2(g'-2g+1) = 22$;
but \href{https://www.lmfdb.org/Variety/Abelian/Fq/?q=3&g=3&jacobian=yes&curve_point_count=10}{this LMFDB query} shows that there are exactly two isomorphism classes of curves $C$ of genus 3 over $\FF_3$ with $\#C(\FF_3) = 10$, and both satisfy $\#C(\FF_{3^2}) \leq 12$.

%\tdi{Run cyclic covers for $q=3, g=3$, and $g' \leq 7$.}

For $g=2$, Table~\ref{tab:datos} gives $8 \geq \#C(\FF_3) \geq 3(g'-g) - 2$, so $g' - g \leq 3$ and $g' \leq 5$.
We rule out $g'=5$ as follows
by noting that $8 \geq \#C(\FF_3) \geq 7$ and arguing as follows.
\begin{itemize}
    \item 
    If $\#C(\FF_3) = 8$, then $\#C(\FF_{3^2}) \geq 15$.
However, this \href{https://www.lmfdb.org/Variety/Abelian/Fq/?q=3&g=2&jacobian=yes&curve_point_count=8}{LMFDB query} shows that the possible isogeny classes for the Jacobian of a curve $C$ with $\#C(\FF_3) = 8$ 
are \avlink{2.3.e_i}, \avlink{2.3.e_j}, \avlink{2.3.e_k}, and this shows that  $\#C(\FF_{3^2}) \leq 14$.
\item 
If $\#C(\FF_3) = 7$, then 
$\#C'(\FF_3) = 0$ and so the covering $C' \to C$ cannot have any $\mathbb{F}_3$-rational ramification points.
As in \cite[(7.6), (7.7)]{Kedlaya2022}, we must have
$2 \#C(\FF_3) \leq \#C'(\FF_{3^2}) = \#C(\FF_{3^2}) - T_{A,3^2}$
and so $\#C(\FF_{3^2}) \geq 2\#C(\FF_3) + T_{A,3^2} \geq 17$.
However, this \href{https://www.lmfdb.org/Variety/Abelian/Fq/?q=3&g=2&jacobian=yes&curve_point_count=7}{LMFDB query} shows that the possible isogeny classes for the Jacobian of a curve $C$ with $\#C(\FF_3) = 7$ 
are \avlink{2.3.d_f}, \avlink{2.3.d_g}, \avlink{2.3.d_h}, and this shows that  $\#C(\FF_{3^2}) \leq 16$.
\end{itemize}

%\tdi{Run cyclic covers for $q=3, g=2$, and $g' \leq 5$.}

%%%%%%%%%%%%%%%%%%%%%%%%%%%%%%%%%%

\subsubsection{$g \geq 2$ and $d = 3$} 
\label{sec:subsec-q=3-geometric-d=3-case}

In the case that $B \sim B_1$, by the given conditions, Riemann--Hurwitz, \Cref{eq:prym-trace-q3}, and Lemma~\ref{lemma:points-bounds},
\[6g - 7 \leq 3(g'-g)-1 = T_{A,3} \leq \#C(\FF_3) \leq 1.91g + 5.52,\]
which implies $g=2,3$. 
For $g=3$, we have $\#C(\FF_3) \leq 10$ from Table~\ref{table:many-points},
but this implies $g' \leq 6$ which is impossible.
For $g = 2$, Table~\ref{table:many-points} indicates $\#C(\FF_3) \leq 8$,
which implies that $4\leq g'\leq 5$; note that here $g' = 5$ forces $\#C(\FF_3) = 8$.

When $B \sim B_2$, by the given conditions, Riemann--Hurwitz, \Cref{eq:prym-trace-q3}, and Lemma~\ref{lemma:points-bounds},
\[6g - 8 \leq 3(g'-g)-2 = T_{A,3} \leq \#C(\FF_3) \leq 1.91g + 5.52,\]
which implies $g=2,3$. Again, we see that $4 \leq g' \leq 5$ for $g = 2$, and $g' = 7$ for $g = 3$;
in this case, $g'=5$ forces $\#C(\FF_3) \geq 7$.

By Lemma~\ref{lem:prime-deg-congruence-geometric}, $F'/F$ cannot be cyclic. 
As in section~\ref{subsec:noncyclic-description}, 
the quadratic resolvent of its Galois closure $F''/F$ is purely geometric.
The Prym of the quadratic resolvent is an elliptic curve $E'$ satisfying
\[
0 \leq \#C''(\FF_3) = \#C(\FF_3) - 2 T_{A,3} - T_{E',3} \leq \#C(\FF_3) - 2(3(g'-g) -2) - T_{E',3},
\]
and $|T_{E',3}| \leq 3$.
For $g = 3$, this yields $10 \geq \#C(\FF_3) \geq 17$; for $g=2, g'=5$, we get $8 \geq \#C(\FF_3) \geq 11$. 
Hence $g=2, g'=4$.

We now have 
$5 \leq 8 + T_{E',3} \leq \#C(\FF_3) \leq 8$, so $T_{E',3} \leq 0$.
Also, $C'''$ admits a cyclic \'etale cubic cover which is not a base extension from $C$, 
so by class field theory, $\#E'(\FF_3) \equiv 0 \pmod{3}$. Together this yields $T_{E',3} = -2$
and $\#C(\FF_3) \geq 2 T_{A,3} -T_{E',3} \geq 6$; using the following lemma, we deduce that $\#C(\FF_3) = 6$ and $T_{A,3} = 4$.

\begin{lemma} \label{lem:relative-quadratic-twist-genus-2}
Let $C$ be a genus $2$ curve over $\FF_3$ admitting an \'etale double cover whose Prym variety $E'$
satisfies $|T_{E',3}| = 2$. Then $C$ appears in Table~\ref{tab:geometric-q3-g2-d2-g'3};
in particular, $\#C(\FF_3) \leq 6$.
\end{lemma}
\begin{proof}
If $T_{E',3} = 2$, the the given double cover has relative quadratic twist 2.
If $T_{E',3} = -2$, then the quadratic twist of $E'$ is the Prym variety of the relative quadratic twist of the cover, so the latter has relative class number 2.
\end{proof}

Since $C$ admits an \'etale cover, $\#J(\FF_3)$ is even. Since we also have $\#C(\FF_3) = 6$,
we may use \href{https://www.lmfdb.org/Variety/Abelian/Fq/?q=3&g=2&jacobian=yes&curve_point_count=6}{this LMFDB query} to see that $J$ belongs to one of \avlink{2.3.c_c}, \avlink{2.3.c_e}, \avlink{2.3.c_g};
this restricts $C$ to a list of 6 possibilities.
We enumerate towers of coverings $C'' \to C''' \to C$ where $C''' \to C$ is \'etale of degree 2 with Prym isogenous to $E'$ 
(per Lemma~\ref{lem:relative-quadratic-twist-genus-2} these can be found in Table~\ref{tab:geometric-q3-g2-d2-g'3})
and $C'' \to C'''$ is cyclic \'etale of degree 3,
looking for cases where the latter has relative class number 4; 
the code for this can be found in the notebook \href{https://github.com/sarangop1728/twice-class-number/blob/main/q3_geometric.ipynb}{\texttt{q3_geometric.ipynb}}
listed in \Cref{tab:meta-table}.
We find exactly one instance, which does give rise to an extension
as described in the following lemma.

\begin{lemma} \label{lem:noncyclic cover}
There exists a noncyclic cubic extension $F'/F$ of function fields over $\FF_3$ with relative class number $2$,
corresponding to a map between the curves
\begin{align*}
C\colon& y^2 = x^6 + x^5 + 2x^4 + 2x^2 + 2x + 1, \\
C'\colon& y^2=x^9+x^8+x^7+2x^3+x^2+2x.
\end{align*}
\end{lemma}
\begin{proof}
The Jacobian $J$ of $C$ belongs to \avlink{2.3.c_c}. In particular, $\#J(\FF_3)$ is not divisible by 3, so $C$ does not admit an \'etale cyclic triple cover.

From the computation described above, we obtain an \'etale double cover $C''' \to C$ with Prym isogenous to $E'$ and a unique \'etale cyclic triple cover $C'' \to C'''$ with Prym isogenous to $B_2^2$.
This uniqueness ensures that $C'' \to C$ is Galois;
this cover cannot be abelian because $C$ admits no \'etale cyclic triple cover, so its Galois group is $S_3$.
Set $C' := C''/S_2$; then $C' \to C$ is a triple cover with Prym isogenous to $B_2$.
In particular, the Jacobian $J'$ of $C'$ belongs to \avlink{4.3.ac_c_c_ao}, which contains a unique Jacobian;
this confirms the equation for $C'$.

For completeness, we also exhibit explicit equations for the map $C' \to C$ (provided by Everett Howe):
\begin{align*}
    x &\mapsto \frac{x^3+2x+2}{x^3+2x^2+1}, \\
    y &\mapsto \frac{x^4+x^3+x^2+2x+1}{x^9+2x^6+1} y. \qedhere
\end{align*}
\end{proof}

\subsubsection{$g \geq 2$ and $d > 3$}
\label{sec:subsec-q=3-geometric-d=4-case}

As in section \ref{sec:subsec-q=3-geometric-d=3-case}, we have $g=2$
and 
\[
7 \leq 3(d-1)-2 \leq 3(g'-g)-2 \leq T_{A,3} \leq \#C(\FF_3) \leq 8.
\]
It follows that $d=4$, $g = 2$, $g' = 5$, and $F'/F$ is unramified;
moreover, $\#C(\FF_3) \in \{7,8\}$, which restricts $J$ to one of six isogeny classes listed in section \ref{sec:subsec-q=3-geometric-d=2-case}.
We need to show that no extensions occur in this case.

We first rule out the possibility that $F'/F$ admits an intermediate subfield $F''$;
note that this includes the case where $F'/F$ is cyclic.
In this case, the extension $F''/F$ would have relative class number 1 or 2;
however, the former is impossible because \cite[Table~4]{Kedlaya2022} implies $\#C(\FF_3) \in \{3,5\}$,
while the latter is impossible since Lemma~\ref{lem:relative-quadratic-twist-genus-2} implies $\#C(\FF_3) \leq 6$.

We assume hereafter that $F'/F$ has no intermediate subfield,
and set notation as in section~\ref{subsec:noncyclic-description}; in particular, $F''/F$ is now the Galois closure of $F'/F$. As in \cite[Lemma~4.5]{Kedlaya2024-II}, we can use the point counts of $C$ and $C'$ to see that the Galois group of $F''/F$ must contain a 4-cycle, and so must equal $S_4$.

If the quadratic resolvent $F'''/F$ is constant, then $C_{\FF_9}$ admits a cyclic cubic \'etale cover which is not a base extension from $C$, so as in \cite[Remark~5.2]{Kedlaya2024-II} we have
$\#J(\FF_9)/\#J(\FF_3) \equiv 0 \pmod{3}$; this excludes the isogeny classes \avlink{2.3.d_g}, \avlink{2.3.d_h}, \avlink{2.3.e_i}, \avlink{2.3.e_k}. 
We handle the remaining cases using splitting sequences as in \cite{Kedlaya2024-II}.
\begin{itemize}
    \item If $J$ is in \avlink{2.3.d_f}, then
\[
(\#C(\FF_{3^i}))_{i=1}^2 = (7, 11), \qquad
(\#C'(\FF_{3^i}))_{i=1}^4 = (0, 8, 33, 40).
\]
Because $F'''/F$ is constant, the splitting type of a place of $F$ of odd (resp. even) degree must correspond to the cycle type of an odd (resp. even) permutation in $S_4$.
By looking at places of degree 1 and 2 in $F$ and $F'$,
we see that the splitting sequence must begin $\{1^4 \times 7\}, \{1^4, 2^2\}$; 
however, this creates too many places of degree 4.
\item 
If $J$ is in \avlink{2.3.e_j}, similar logic about splitting sequences rules out $\#C'(\FF_3) = 1$, so we have
\[
(\#C(\FF_{3^i}))_{i=1}^3 = (8,12,20), \qquad
(\#C'(\FF_{3^i}))_{i=1}^3 = (0,8,30).
\]
This time, the splitting sequence begins $\{4 \times 8\}$, 
so the 10 degree-3 places of $F'$ sit over the 4 degree-3 places of $F$; however, the latter all have splitting types $4$ or $2+1^2$ which leaves room for at most 8 degree-3 places of $F'$.
\end{itemize}

If the quadratic resolvent is purely geometric, then it corresponds to an \'etale double cover $C''' \to C$. Let $E'$ be the Prym of this cover; again as in \cite[Remark~5.2]{Kedlaya2024-II}, the existence of an \'etale cyclic triple cover of $C'''$ not induced from $C$ ensures that $\#E'(\FF_3) \equiv 0 \pmod{3}$. By Lemma~\ref{lem:relative-quadratic-twist-genus-2} we cannot have $T_{E',3} = -2$; this leaves only the option $T_{E',3} = 1$.
Since $C''' \to C$ is a Galois cover, $\#C'''(\FF_3) = \#C(\FF_3) - T_{E',3}$ must be even, forcing
$\#C(\FF_3) = 7$.
By class field theory, $\#J(\FF_3)$ must be even; this restricts $J$ to the single isogeny class \avlink{2.3.d_g}, so
\[
\#C(\FF_3) = 7, \qquad \#C'''(\FF_3) = 6, \qquad (\#C'(\FF_{3^i}))_{i=1}^2 = (0, 10).
\]
The splitting sequence begins $\{4 \times 2, 2^2 \times 3\}$, but this forces the contradiction $\#C'(\FF_{3^2}) \geq 12$.

\section{Explicit calculations for geometric covers}
\label{sec:explicit calculations for geometric}

In light of Proposition~\ref{prop:summary}, the explicit resolution of the relative class number 2 problem for function fields with $q > 2$ is reduced to an exhaust over quadratic covers for the particular values of $q,g,g'$ indicated in Table~\ref{table:bounds table} for which $g \geq 2$. 
We perform this exhaust using explicit class field theory as implemented in Magma, as in \cite{Kedlaya2022}; we report the result in a series of tables listed in Table~\ref{table:bounds table}.
In some cases, we use constraints on $\#C(\FF_q)$ and $\#C(\FF_{q^2})$ derived in our earlier analysis;
these are also cited in Table~\ref{table:bounds table}.

This exhaust requires as input a list of systems of polynomial equations defining isomorphism class representatives for the curves of genus $g$ over $\FF_q$. As of this writing, this data is available in LMFDB for all pairs $(q,d)$ that we require; however, some notes on sourcing are in order.
\begin{itemize}
\item
For $q=5,g=2$, all of the curves are hyperelliptic, so the algorithm of Howe \cite{Howe2024} can be used to enumerate isomorphism class representatives.
\item
For $q=3,g=4$, a census of curves was made by the authors of \cite{BergstromFaberPayne2024}, who generously shared their data with LMFDB.
\end{itemize}

\begin{table}[ht]
    \setlength{\arrayrulewidth}{0.2mm} 
    \setlength{\tabcolsep}{5pt}
    \renewcommand{\arraystretch}{1.2}
    \centering
    \begin{tabular}{|c|c|c|c|c|c|c|c|}
        \hline
        \rowcolor{headercolor} 
        $q$ & $g$ & $g'$ & Bound & Printed table & File & Code \\
        \hline
        5 & 2 & 3 & - & \Cref{tab:geometric-q5-g2-d2-g'3} & \href{https://github.com/sarangop1728/twice-class-number/blob/main/q5_geometric.ipynb}{\texttt{tables/table_q5_g2_d2_gg3.txt}}& \href{https://github.com/sarangop1728/twice-class-number/blob/main/q5_geometric.ipynb}{\texttt{q5_geometric.ipynb}}\\
        \hline
        4 & 2 & 3 & - & \Cref{tab:geometric-q4-g2-d2-g'3} & \href{https://github.com/sarangop1728/twice-class-number/blob/main/tables/table_q4_g2_d2_gg3.txt}{\texttt{tables/table_q4_g2_d2_gg3.txt}} & \multirow{3}{*}{\href{https://github.com/sarangop1728/twice-class-number/blob/main/q4_geometric.ipynb}{\texttt{q4_geometric.ipynb}}} \\
        4 & 2 & 4 & \eqref{eq:constraint-q4-g2-gg4} & \Cref{tab:geometric-q4-g2-d2-g'4}  & \href{https://github.com/sarangop1728/twice-class-number/blob/main/tables/table_q4_g2_d2_gg4.txt}{\texttt{tables/table_q4_g2_d2_gg4.txt}} & \\
%         4 & 2 & 3 & 4 & - & - & \\
        4 & 3 & 5 & \eqref{eq:constraint-q4-g3-gg5} & \Cref{tab:geometric-q4-g3-d2-g'5} & \href{https://github.com/sarangop1728/twice-class-number/blob/main/tables/table_q4_g3_d2_gg5.txt}{\texttt{tables/table_q4_g3_d2_gg5.txt}} & \\
        \hline
        3 & 2 & 3 & - & \Cref{tab:geometric-q3-g2-d2-g'3} & \href{https://github.com/sarangop1728/twice-class-number/blob/main/tables/table_q3_g2_d2_gg3.txt}{\texttt{tables/table_q3_g2_d2_gg3.txt}} & \multirow{5}{*}{\href{https://github.com/sarangop1728/twice-class-number/blob/main/q3_geometric.ipynb}{\texttt{q3_geometric.ipynb}}} \\
        3 & 2 & 4 & - & \Cref{tab:geometric-q3-g2-d2-g'4} & \href{https://github.com/sarangop1728/twice-class-number/blob/main/tables/table_q3_g2_d2_gg4.txt}{\texttt{tables/table_q3_g2_d2_gg4.txt}} & \\

        3 & 3 & 5 & - & \Cref{tab:geometric-q3-g3-d2-g'5} & \href{https://github.com/sarangop1728/twice-class-number/blob/main/tables/table_q3_g3_d2_gg5.txt}{\texttt{tables/table_q3_g3_d2_gg5.txt}} & \\
        3 & 3 & 6 & \eqref{eq:constraint-q3-g3-gg6} & - & - & \\

        3 & 4 & 7 & \eqref{eq:constraint-q3-g4-gg7} & \Cref{tab:geometric-q3-g4-d2-g'7} & \href{https://github.com/sarangop1728/twice-class-number/blob/main/tables/table_q3_g4_d2_gg7.txt}{\texttt{tables/table_q3_g4_d2_gg7.txt}} & \\
        \hline
    \end{tabular}
    \caption{Tuples $(q,g,g')$ for which we explicitly tabulate geometric extensions, with links to the corresponding tables.}
    \label{tab:meta-table}
\end{table}

\newpage
\begin{table}[ht] 
    \setlength{\arrayrulewidth}{0.2mm}
    \setlength{\tabcolsep}{5pt} 
    \renewcommand{\arraystretch}{1.2}
    \centering 
    \begin{tabular}{|c|c|}
        \hline
        \rowcolor{headercolor}
        $J$ & $F$ \\
        \hline
        \avlink{2.5.c_e} & $y^{2} + 4 x^{6} + 2 x^{5} + 2 x^{4} + x^{3} + 2 x^{2} + 4 x + 3$ \\
        \hline
        \avlink{2.5.ac_e} & $y^{2} + 3 x^{6} + 4 x^{5} + 4 x^{4} + 2 x^{3} + 4 x^{2} + 3 x + 1$ \\
        \hline
        \avlink{2.5.c_i} & $y^{2} + 4 x^{6} + x^{5} + 4 x^{4} + 4 x^{2} + 3 x + 3$ \\
        \hline
        \avlink{2.5.ac_i} & $y^{2} + 3 x^{6} + 2 x^{5} + 3 x^{4} + 3 x^{2} + x + 1$ \\
        \hline
        \avlink{2.5.ac_k} & $y^{2} + 4 x^{6} + x^{5} + 3 x^{3} + x + 4$ \\
        \hline
        \avlink{2.5.c_k} & $y^{2} + 3 x^{6} + 2 x^{5} + x^{3} + 2 x + 3$ \\
        \hline
        \avlink{2.5.a_e} & $y^{2} + 4 x^{5} + x^{4} + x^{3} + x^{2} + x + 4$ \\
        \hline
        \avlink{2.5.a_e} & $y^{2} + 3 x^{5} + 2 x^{4} + 2 x^{3} + 2 x^{2} + 2 x + 3$ \\
        \hline
        \avlink{2.5.ac_c} & $y^{2} + 4 x^{5} + 3 x^{3} + 3 x + 3$ \\
        \hline
        \avlink{2.5.c_c} & $y^{2} + 3 x^{5} + x^{3} + x + 1$ \\
        \hline
        \avlink{2.5.c_g} & $y^{2} + 4 x^{5} + 3 x^{4} + 2 x^{3} + 3 x^{2} + 2 x + 1$ \\
        \hline
        \avlink{2.5.ac_g} & $y^{2} + 3 x^{5} + x^{4} + 4 x^{3} + x^{2} + 4 x + 2$ \\
        \hline
        \avlink{2.5.ac_k} & $y^{2} + 4 x^{5} + 2 x^{4} + 4 x^{3} + 2 x^{2} + 4 x$ \\
        \hline
        \avlink{2.5.c_k} & $y^{2} + 3 x^{5} + 4 x^{4} + 3 x^{3} + 4 x^{2} + 3 x$ \\
        \hline
        \avlink{2.5.a_a} & $y^{2} + 4 x^{6} + 2 x^{5} + 3 x + 1$ \\
        \hline
        \avlink{2.5.ac_g} & $y^{2} + 4 x^{5} + 4 x^{4} + x^{3} + 2 x^{2} + 2 x$ \\
        \hline
        \avlink{2.5.c_g} & $y^{2} + 3 x^{5} + 3 x^{4} + 2 x^{3} + 4 x^{2} + 4 x$ \\
        \hline
        \avlink{2.5.a_c} & $y^{2} + 4 x^{5} + 4 x^{4} + 2 x^{2} + 4 x$ \\
        \hline
        \avlink{2.5.a_c} & $y^{2} + 4 x^{5} + 4 x^{4} + x^{3} + 3 x^{2} + x$ \\
        \hline
        \avlink{2.5.a_k} & $y^{2} + 4 x^{5} + 4 x$ \\
        \hline
        \avlink{2.5.a_k} & $y^{2} + 4 x^{5} + 4 x$ \\
        \hline
        \avlink{2.5.a_g} & $y^{2} + 4 x^{6} + 3 x^{4} + x^{2} + 2$ \\
        \hline
        \avlink{2.5.a_g} & $y^{2} + 4 x^{6} + 3 x^{4} + x^{2} + 2$ \\
        \hline
        \avlink{2.5.c_c} & $y^{2} + 4 x^{6} + 4 x^{5} + 4 x^{4} + x^{2} + x + 1$ \\
        \hline
        \avlink{2.5.ac_c} & $y^{2} + 3 x^{6} + 3 x^{5} + 3 x^{4} + 2 x^{2} + 2 x + 2$ \\
        \hline
        \avlink{2.5.a_ac} & $y^{2} + 4 x^{6} + 4 x^{5} + 3 x^{4} + x^{2} + x + 2$ \\
        \hline
    \end{tabular} 
\caption{Purely geometric covers of type $(q,g,d,g') = (5,2,2,3)$.} 
\label{tab:geometric-q5-g2-d2-g'3} 
\end{table}

\begin{table}[ht] 
    \setlength{\arrayrulewidth}{0.2mm}
    \setlength{\tabcolsep}{5pt} 
    \renewcommand{\arraystretch}{1.2}
    \centering 
    \begin{tabular}{|c|c|}
        \hline
        \rowcolor{headercolor}
        $J$ & $F$ \\
        \hline
        \avlink{2.4.ac_j} & $y^{2} + \left(x^{2} + x\right) y + x^{5} + x^{3} + x^{2} + x$ \\
        \hline
        \avlink{2.4.ac_j} & $y^{2} + \left(x^{2} + x\right) y + x^{5} + x^{3} + x^{2} + x$ \\
        \hline
        \avlink{2.4.ac_j} & $y^{2} + \left(x^{2} + x\right) y + x^{5} + x^{3} + x^{2} + x$ \\
        \hline
        \avlink{2.4.c_j} & $y^{2} + \left(x^{2} + x\right) y + x^{5} + a x^{4} + x^{3} + \left(a + 1\right) x^{2} + x$ \\
        \hline
        \avlink{2.4.c_j} & $y^{2} + \left(x^{2} + x\right) y + x^{5} + a x^{4} + x^{3} + \left(a + 1\right) x^{2} + x$ \\
        \hline
        \avlink{2.4.c_j} & $y^{2} + \left(x^{2} + x\right) y + x^{5} + a x^{4} + x^{3} + \left(a + 1\right) x^{2} + x$ \\
        \hline
        \avlink{2.4.ac_d} & $y^{2} + \left(x^{2} + x + a\right) y + a x^{5} + a x^{4} + a x^{2} + a$ \\
        \hline
        \avlink{2.4.ac_d} & $y^{2} + \left(x^{2} + x + a + 1\right) y + \left(a + 1\right) x^{5} + a x^{4} + a x^{2} + 1$ \\
        \hline
        \avlink{2.4.ac_b} & $y^{2} + \left(x^{2} + x\right) y + x^{5} + \left(a + 1\right) x^{2} + a x$ \\
        \hline
        \avlink{2.4.a_ab} & $y^{2} + \left(x^{2} + x\right) y + a x^{5} + a x^{3} + x^{2} + x$ \\
        \hline
        \avlink{2.4.a_ab} & $y^{2} + \left(x^{2} + x\right) y + \left(a + 1\right) x^{5} + \left(a + 1\right) x^{3} + x^{2} + x$ \\
        \hline
        \avlink{2.4.a_b} & $y^{2} + \left(x^{2} + x + a\right) y + a x^{5} + a x^{3} + x^{2} + a$ \\
        \hline
        \avlink{2.4.a_b} & $y^{2} + \left(x^{2} + x + a + 1\right) y + \left(a + 1\right) x^{5} + \left(a + 1\right) x^{3} + x^{2} + a + 1$ \\
        \hline
        \avlink{2.4.a_d} & $y^{2} + \left(x^{2} + x\right) y + \left(a + 1\right) x^{5} + \left(a + 1\right) x^{3} + a x^{2} + a x$ \\
        \hline
        \avlink{2.4.a_d} & $y^{2} + \left(x^{2} + x\right) y + \left(a + 1\right) x^{5} + \left(a + 1\right) x^{3} + a x^{2} + a x$ \\
        \hline
        \avlink{2.4.a_d} & $y^{2} + \left(x^{2} + x\right) y + a x^{5} + a x^{3} + \left(a + 1\right) x^{2} + \left(a + 1\right) x$ \\
        \hline
        \avlink{2.4.a_d} & $y^{2} + \left(x^{2} + x\right) y + a x^{5} + a x^{3} + \left(a + 1\right) x^{2} + \left(a + 1\right) x$ \\
        \hline
        \avlink{2.4.a_h} & $y^{2} + \left(x^{2} + x + a\right) y + x^{5} + x^{3} + x^{2} + a x + a$ \\
        \hline
        \avlink{2.4.a_h} & $y^{2} + \left(x^{2} + x + a\right) y + x^{5} + a x^{4} + x^{3} + \left(a + 1\right) x^{2} + a x + a + 1$ \\
        \hline
        \avlink{2.4.ac_f} & $y^{2} + \left(x^{2} + x + a + 1\right) y + x^{5} + a x^{4} + a x^{3} + a x^{2} + \left(a + 1\right) x + 1$ \\
        \hline
        \avlink{2.4.ac_f} & $y^{2} + \left(x^{2} + x + a\right) y + x^{5} + a x^{4} + \left(a + 1\right) x^{3} + a x^{2} + a x + a$ \\
        \hline
        \avlink{2.4.ac_h} & $y^{2} + \left(x^{2} + x + a + 1\right) y + a x^{5} + x^{3} + a x + a$ \\
        \hline
        \avlink{2.4.ac_h} & $y^{2} + \left(x^{2} + x + a\right) y + \left(a + 1\right) x^{5} + x^{3} + \left(a + 1\right) x + a + 1$ \\
        \hline
        \avlink{2.4.c_b} & $y^{2} + \left(x^{2} + x\right) y + x^{5} + \left(a + 1\right) x^{4} + a x$ \\
        \hline
        \avlink{2.4.c_d} & $y^{2} + \left(x^{2} + x + a\right) y + a x^{5} + a + 1$ \\
        \hline
        \avlink{2.4.c_d} & $y^{2} + \left(x^{2} + x + a + 1\right) y + \left(a + 1\right) x^{5} + a$ \\
        \hline
        \avlink{2.4.c_f} & $y^{2} + \left(x^{2} + x + a + 1\right) y + x^{5} + a x^{3} + \left(a + 1\right) x + a$ \\
        \hline
        \avlink{2.4.c_f} & $y^{2} + \left(x^{2} + x + a\right) y + x^{5} + \left(a + 1\right) x^{3} + a x + a + 1$ \\
        \hline
        \avlink{2.4.c_h} & $y^{2} + \left(x^{2} + x + a + 1\right) y + a x^{5} + \left(a + 1\right) x^{4} + x^{3} + \left(a + 1\right) x^{2} + a x + a + 1$ \\
        \hline
        \avlink{2.4.c_h} & $y^{2} + \left(x^{2} + x + a\right) y + \left(a + 1\right) x^{5} + a x^{4} + x^{3} + a x^{2} + \left(a + 1\right) x + a$ \\
        \hline
    \end{tabular} 
\caption{Purely geometric covers of type $(q,g,d,g') = (4,2,2,3)$.} 
\label{tab:geometric-q4-g2-d2-g'3} 
\end{table}

\begin{table}[ht] 
    \setlength{\arrayrulewidth}{0.2mm}
    \setlength{\tabcolsep}{5pt} 
    \renewcommand{\arraystretch}{1.2}
    \centering 
    \begin{tabular}{|c|c|}
        \hline
        \rowcolor{headercolor}
        $J$ & $F$ \\
        \hline
        \avlink{2.4.d_i} & $y^{2} + x y + x^{5} + x^{3} + x$ \\
        \hline
    \end{tabular} 
\caption{Purely geometric covers of type $(q,g,d,g') = (4,2,2,4)$.} 
\label{tab:geometric-q4-g2-d2-g'4} 
\end{table}

\begin{table}[ht] 
    \setlength{\arrayrulewidth}{0.2mm}
    \setlength{\tabcolsep}{5pt} 
    \renewcommand{\arraystretch}{1.2}
    \centering 
    \begin{tabular}{|c|c|}
        \hline
        \rowcolor{headercolor}
        $J$ & $F$ \\
        \hline
        \avlink{3.4.c_f_e} & $y^{2} + \left(a x^{3} + a x^{2}\right) y + x^{7} + x^{6} + \left(a + 1\right) x^{2} + \left(a + 1\right) x$ \\
        \hline
        \avlink{3.4.c_f_u} & $y^{2} + \left(a x^{3} + a x^{2}\right) y + \left(a + 1\right) x^{7} + \left(a + 1\right) x^{2} + \left(a + 1\right) x$ \\
        \hline
        \avlink{3.4.c_h_m} & $y^{2} + \left(\left(a + 1\right) x^{4} + a x^{3} + a x^{2}\right) y + \left(a + 1\right) x + a$ \\
        \hline
        \avlink{3.4.c_h_m} & $y^{2} + \left(\left(a + 1\right) x^{4} + a x^{3} + a x^{2}\right) y + \left(a + 1\right) x^{2} + \left(a + 1\right) x$ \\
        \hline
        \avlink{3.4.e_n_bc} & $y^{2} + \left(\left(a + 1\right) x^{4} + a x^{3} + a x^{2}\right) y + x^{3} + a x^{2} + \left(a + 1\right) x + a$ \\
        \hline
    \end{tabular} 
\caption{Purely geometric covers of type $(q,g,d,g') = (4,3,2,5)$.} 
\label{tab:geometric-q4-g3-d2-g'5} 
\end{table}

\begin{table}[ht] 
    \setlength{\arrayrulewidth}{0.2mm}
    \setlength{\tabcolsep}{5pt} 
    \renewcommand{\arraystretch}{1.2}
    \centering 
    \begin{tabular}{|c|c|}
        \hline
        \rowcolor{headercolor}
        $J$ & $F$ \\
        \hline
        \avlink{2.3.a_ac} & $y^{2} + 2 x^{5} + x$ \\
        \hline
        \avlink{2.3.a_ac} & $y^{2} + 2 x^{5} + x$ \\
        \hline
        \avlink{2.3.a_ac} & $y^{2} + 2 x^{5} + x$ \\
        \hline
        \avlink{2.3.a_ac} & $y^{2} + 2 x^{5} + x$ \\
        \hline
        \avlink{2.3.a_a} & $y^{2} + 2 x^{5} + 1$ \\
        \hline
        \avlink{2.3.a_a} & $y^{2} + x^{5} + 2$ \\
        \hline
        \avlink{2.3.a_c} & $y^{2} + 2 x^{5} + 2 x$ \\
        \hline
        \avlink{2.3.a_c} & $y^{2} + 2 x^{5} + 2 x$ \\
        \hline
        \avlink{2.3.a_e} & $y^{2} + 2 x^{6} + x^{5} + 2 x + 1$ \\
        \hline
        \avlink{2.3.ac_c} & $y^{2} + x^{5} + 2 x^{4} + 2 x^{3} + 2 x^{2} + 2 x + 1$ \\
        \hline
        \avlink{2.3.ac_c} & $y^{2} + x^{6} + x^{5} + 2 x^{4} + 2 x^{2} + 2 x + 1$ \\
        \hline
        \avlink{2.3.ac_e} & $y^{2} + x^{6} + 2 x^{5} + x^{4} + x + 2$ \\
        \hline
        \avlink{2.3.ac_g} & $y^{2} + x^{5} + 2 x^{4} + x^{3} + 2 x^{2} + x$ \\
        \hline
        \avlink{2.3.ac_g} & $y^{2} + x^{6} + x^{5} + 2 x^{4} + 2 x^{3} + 2 x^{2} + x + 1$ \\
        \hline
        \avlink{2.3.c_c} & $y^{2} + 2 x^{5} + x^{4} + x^{3} + x^{2} + x + 2$ \\
        \hline
        \avlink{2.3.c_c} & $y^{2} + 2 x^{6} + 2 x^{5} + x^{4} + x^{2} + x + 2$ \\
        \hline
        \avlink{2.3.c_e} & $y^{2} + 2 x^{6} + x^{5} + 2 x^{4} + 2 x + 1$ \\
        \hline
        \avlink{2.3.c_g} & $y^{2} + 2 x^{5} + x^{4} + 2 x^{3} + x^{2} + 2 x$ \\
        \hline
        \avlink{2.3.c_g} & $y^{2} + 2 x^{6} + 2 x^{5} + x^{4} + x^{3} + x^{2} + 2 x + 2$ \\
        \hline
    \end{tabular} 
\caption{Purely geometric covers of type $(q,g,d,g') = (3,2,2,3)$.} 
\label{tab:geometric-q3-g2-d2-g'3} 
\end{table}

\begin{table}[ht] 
    \setlength{\arrayrulewidth}{0.2mm}
    \setlength{\tabcolsep}{5pt} 
    \renewcommand{\arraystretch}{1.2}
    \centering 
    \begin{tabular}{|c|c|}
        \hline
        \rowcolor{headercolor}
        $J$ & $F$ \\
        \hline
        \avlink{2.3.b_c} & $y^{2} + x^{6} + x^{5} + 2 x^{3} + x^{2} + x + 2$ \\
        \hline
        \avlink{2.3.c_e} & $y^{2} + 2 x^{6} + x^{5} + 2 x^{4} + 2 x + 1$ \\
        \hline
        \avlink{2.3.c_e} & $y^{2} + 2 x^{6} + x^{5} + 2 x^{4} + 2 x + 1$ \\
        \hline
        \avlink{2.3.d_g} & $y^{2} + 2 x^{6} + 2 x + 2$ \\
        \hline
    \end{tabular} 
\caption{Purely geometric covers of type $(q,g,d,g') = (3,2,2,4)$.} 
\label{tab:geometric-q3-g2-d2-g'4} 
\end{table}

\begin{table}[ht] 
    \setlength{\arrayrulewidth}{0.2mm}
    \setlength{\tabcolsep}{5pt} 
    \renewcommand{\arraystretch}{1.2}
    \centering 
    \begin{tabular}{|c|c|}
        \hline
        \rowcolor{headercolor}
        $J$ & $F$ \\
        \hline
        \avlink{3.3.a_ab_g} & $x y^{3} + \left(x + 1\right) y^{2} + y + x^{3}$ \\
        \hline
        \avlink{3.3.a_b_ae} & $y^{2} + 2 x^{7} + 2 x^{6} + 2 x^{3} + x^{2} + x + 1$ \\
        \hline
        \avlink{3.3.a_b_e} & $y^{2} + x^{7} + x^{6} + x^{3} + 2 x^{2} + 2 x + 2$ \\
        \hline
        \avlink{3.3.a_b_e} & $x y^{3} + \left(2 x^{2} + 2\right) y^{2} + \left(x + 1\right) y + x^{4}$ \\
        \hline
        \avlink{3.3.a_d_a} & $y^{2} + 2 x^{8} + 2 x^{7} + 2 x^{6} + x^{2} + 2 x + 1$ \\
        \hline
        \avlink{3.3.a_d_a} & $y^{2} + 2 x^{7} + 2 x^{5} + x^{4} + x^{3} + 2 x^{2} + 1$ \\
        \hline
        \avlink{3.3.a_d_a} & $y^{3} + x^{3} y + 2 x^{4} + x$ \\
        \hline
        \avlink{3.3.b_b_ac} & $y^{2} + 2 x^{7} + x^{5} + 2 x^{4} + 2 x^{2} + 2$ \\
        \hline
        \avlink{3.3.b_b_g} & $y^{2} + 2 x^{8} + x^{7} + 2 x^{6} + 2 x^{5} + x^{4} + 2 x^{3} + x^{2} + x$ \\
        \hline
        \avlink{3.3.b_d_a} & $y^{2} + x^{8} + 2 x^{7} + x^{2} + x$ \\
        \hline
        \avlink{3.3.b_d_e} & $y^{2} + 2 x^{8} + 2 x^{7} + 2 x^{6} + 2 x^{5} + x^{4} + x^{2} + x$ \\
        \hline
        \avlink{3.3.b_d_g} & $y^{2} + 2 x^{8} + 2 x^{6} + x^{5} + x^{3} + 2 x^{2} + 2 x + 1$ \\
        \hline
        \avlink{3.3.c_d_c} & $y^{3} + x y^{2} + 2 x^{3} y + x^{4} + 2 x^{3} + x^{2} + x$ \\
        \hline
        \avlink{3.3.c_h_i} & $y^{3} + \left(2 x^{2} + 2 x\right) y^{2} + \left(2 x^{2} + x\right) y + x^{4} + 2 x^{3} + x$ \\
        \hline
        \avlink{3.3.d_h_o} & $y^{2} + 2 x^{8} + x^{7} + x^{6} + x^{5} + 2 x^{3} + 2 x^{2} + 2 x$ \\
        \hline
    \end{tabular} 
\caption{Purely geometric covers of type $(q,g,d,g') = (3,3,2,5)$.} 
\label{tab:geometric-q3-g3-d2-g'5} 
\end{table}

\begin{table}[ht] 
    \setlength{\arrayrulewidth}{0.2mm}
    \setlength{\tabcolsep}{5pt} 
    \renewcommand{\arraystretch}{1.2}
    \centering 
    \begin{tabular}{|c|c|}
        \hline
        \rowcolor{headercolor}
        $J$ & $F$ \\
        \hline
        \avlink{4.3.d_i_n_ba} & $y^{2} + x^{9} + 2 x^{8} + x^{5} + x^{2} + x + 2$ \\
        \hline
        \avlink{4.3.d_i_t_bi} & $y^{2} + 2 x^{9} + x^{8} + x^{6} + x^{5} + x^{4} + 2$ \\
        \hline
    \end{tabular} 
\caption{Purely geometric covers of type $(q,g,d,g') = (3,4,2,7)$.} 
\label{tab:geometric-q3-g4-d2-g'7} 
\end{table}

%%%%%%%%%%%%%%%%%%%%%%%%%%%%%%%%%%
\FloatBarrier

\newpage
\bibliography{refs.bib}{}

\providecommand{\bysame}{\leavevmode\hbox to3em{\hrulefill}\thinspace}
\providecommand{\MR}{\relax\ifhmode\unskip\space\fi MR }
\providecommand{\MRhref}[2]{%
  \href{http://www.ams.org/mathscinet-getitem?mr=#1}{#2}
}
\providecommand{\href}[2]{#2}
\begin{thebibliography}{vdGHLR09}

\bibitem[AC23]{AchterCasalaina-Martin23}
Jeff {Achter} and Sebastian {Casalaina-Martin}, \emph{{Putting the {P} back in Prym}}, arXiv e-prints (2023), arXiv:2312.13263.

\bibitem[BFP24]{BergstromFaberPayne2024}
Jonas Bergstr\"om, Carel Faber, and Sam Payne, \emph{Polynomial point counts and odd cohomology vanishing on moduli spaces of stable curves}, Ann. of Math. (2) \textbf{199} (2024), no.~3, 1323--1365. \MR{4740541}

\bibitem[How23]{Howe2023}
Everett~W. Howe, \emph{Deducing information about curves over finite fields from their {W}eil polynomials}, Curves over finite fields---past, present and future, Panor. Synth\`eses, vol.~60, Soc. Math. France, Paris, 2023, pp.~1--36. \MR{4727503}

\bibitem[How25]{Howe2024}
\bysame, \emph{Enumerating hyperelliptic curves over finite fields in quasilinear time}, Research in Number Theory \textbf{11} (2025), no.~1, 1--24.

\bibitem[Kad21]{Kadets2021}
Borys Kadets, \emph{Estimates for the number of rational points on simple abelian varieties over finite fields}, Math. Z. \textbf{297} (2021), no.~1-2, 465--473. \MR{4204701}

\bibitem[Ked22a]{Kedlaya2022}
Kiran~S. Kedlaya, \emph{The relative class number one problem for function fields, {I}}, Res. Number Theory \textbf{8} (2022), no.~4, Paper No. 79, 21. \MR{4493405}

\bibitem[Ked22b]{Kedlaya2024-II}
\bysame, \emph{{The relative class number one problem for function fields, II}}, arXiv e-prints (2022), arXiv:2206.02084.

\bibitem[Ked24]{Kedlaya2024-III}
\bysame, \emph{The relative class number one problem for function fields, {III}}, Lu{C}a{NT}: {LMFDB}, computation, and number theory, Contemp. Math., vol. 796, Amer. Math. Soc., [Providence], RI, [2024] \copyright2024, pp.~55--73. \MR{4732683}

\bibitem[Ked25]{Kedlaya-distribution}
\bysame, \emph{Distribution of frobenius polynomials of abelian varieties of extremely small order}.

\bibitem[KK24]{KadetsKeliher}
Borys {Kadets} and Daniel {Keliher}, \emph{{Exponents of Jacobians and relative class groups}}, arXiv e-prints (2024), arXiv:2410.11962.

\bibitem[LM76]{LeitzelMadan}
James R.~C. Leitzel and Manohar~L. Madan, \emph{Algebraic function fields with equal class number}, Acta Arith. \textbf{30} (1976), no.~2, 169--177. \MR{414523}

\bibitem[MP77]{MadanPal}
Manohar~L. Madan and Sat Pal, \emph{Abelian varieties and a conjecture of {R}. {M}. {R}obinson}, J. Reine Angew. Math. \textbf{291} (1977), 78--91. \MR{439848}

\bibitem[Ser20]{Serre20-rational-points}
Jean-Pierre Serre, \emph{Rational points on curves over finite fields}, Documents Math\'ematiques (Paris) [Mathematical Documents (Paris)], vol.~18, Soci\'et\'e{} Math\'ematique de France, Paris, [2020] \copyright 2020, With contributions by Everett Howe, Joseph Oesterl\'e{} and Christophe Ritzenthaler. \MR{4242817}

\bibitem[Sta74]{Stark}
H.~M. Stark, \emph{Some effective cases of the {B}rauer-{S}iegel theorem}, Invent. Math. \textbf{23} (1974), 135--152. \MR{342472}

\bibitem[vdGHLR09]{manypoints}
Gerard van~der Geer, Everett~W. Howe, Kristin~E. Lauter, and Christophe Ritzenthaler, \emph{Tables of curves with many points}, 2009, [Online; accessed 12 September 2024].

\end{thebibliography}
\bibliographystyle{amsalpha}
\end{document}